\documentclass[12pt]{amsart}
\usepackage{latexsym}
\usepackage{amssymb, amsmath,mathrsfs}
\usepackage[left=2.5cm,right=2.5cm,top=2.5cm, bottom=2.5cm ]{geometry}
\usepackage[OT2,T1]{fontenc}

\usepackage[pagebackref,hypertexnames=false, colorlinks, citecolor=red, linkcolor=red]{hyperref}

\newtheorem{theorem}{Theorem}[section]
\newtheorem{lemma}[theorem]{Lemma}

\theoremstyle{definition}
\newtheorem{definition}[theorem]{Definition}
\newtheorem{remark}[theorem]{Remark}

\DeclareSymbolFont{cyrletters}{OT2}{wncyr}{m}{n}
\DeclareMathSymbol{\Sha}{\mathalpha}{cyrletters}{"58}

\begin{document}

\title{Well-Localized Operators on Matrix Weighted $L^2$ Spaces}
\date{\today}

\author[K. Bickel]{Kelly Bickel$^\dagger$}
\address{Kelly Bickel, Department of Mathematics\\
Bucknell University\\
701 Moore Ave\\
Lewisburg, PA 17837}
\email{kelly.bickel@bucknell.edu}
\thanks{$\dagger$ Research supported in part by National Science Foundation
DMS grants \# 0955432 and \#1448846.}

\author[B. D. Wick]{Brett D. Wick$^\ddagger$}
\address{Brett D. Wick, School of Mathematics\\
Georgia Institute of Technology\\
686 Cherry Street\\
Atlanta, GA USA 30332-0160}
\email{wick@math.gatech.edu}
\urladdr{www.math.gatech.edu/~wick}
\thanks{$\ddagger$ Research supported in part by National Science Foundation
DMS grant \# 0955432, by the Simons Foundation and
by the Mathematisches Forschungsinstitut Oberwolfach.}

\maketitle

\begin{abstract}
Nazarov-Treil-Volberg recently proved an elegant two-weight T1 theorem for ``almost diagonal'' operators that played a key role in the proof of the $A_2$ conjecture for dyadic shifts and related operators. In this paper, we obtain a generalization of their T1 theorem to the setting of matrix weights. Our theorem does differ slightly from the scalar results, a fact attributable almost completely to differences between the scalar and matrix Carleson Embedding Theorems. The main tools include a reduction to the study of well-localized operators, a new system of Haar functions adapted to matrix weights, and  a matrix Carleson Embedding Theorem.
\end{abstract}

\bibliographystyle{plain}

\section{Introduction}
In this paper, the dimension $d$ is fixed and $L^2$ will denote $L^2(\mathbb{R}, \mathbb{C}^d)$, namely the set of vector-valued functions  satisfying
$$
\left\Vert f\right\Vert_{L^2}^2\equiv\int_{\mathbb{R}} \Vert f(x)\Vert^2_{\mathbb{C}^d}\,dx<\infty.
$$
We will be primarily interested in \emph{matrix weights}, $d \times d$ positive definite matrix-valued functions with locally integrable entries. Given such a weight $W,$ let $L^2(W)$ be the set of functions satisfying
$$
\left\Vert f\right\Vert_{L^2(W)}^2\equiv\int_{\mathbb{R}} \left \Vert W^{\frac{1}{2}}(x)f(x)\right \Vert^2_{\mathbb{C}^d}\,dx=\int_{\mathbb{R}} \left\langle W(x)f(x), f(x)\right\rangle_{\mathbb{C}^d}\,dx<\infty.
$$
Given matrix weights $V$ and $W$, a natural question is:~when does a bounded operator $T$ mapping $L^2$ to itself extend to a bounded operator mapping $L^2(W)$ to $L^2(V)$ and what is the norm of $T$ as a map from $L^2(W)$ to $L^2(V)$? 

If we consider the  special one-dimensional case when $V=W=w$, this question has a classical answer. Indeed, a Calder\'on-Zygmund operator $T$ extends to a bounded operator on $L^2(w)$ if and only if $w$ is an $A_2$  Muckenhoupt weight, namely:
\[
 [ w  ] _{A_2} \equiv \sup_{I} \left  \langle w \rangle_I
 \langle w^{-1}\right \rangle_I  < \infty,
\]
where the supremum is taken over all intervals $I$ and $\left \langle w \right \rangle_I \equiv \frac{1}{|I|} \int_I w(x) dx.$  In contrast, the question of the operator norm of $T$ on $L^2(w)$, and its sharp dependence on $[w]_{A_2},$ called the $A_2$ conjecture, remained open for decades. Lacey-Petermichl-Reguera made substantial progress on this question in \cite{lpr} by establishing the sharp bound for dyadic shifts and as a corollary, obtained new proofs of the bound for simple Calder\'on-Zygmund operators including the Hilbert transform, Riesz transforms, and Beurling transform. Their proof rested on an elegant two-weight T1 theorem due to Nazarov-Treil-Volberg \cite{ntv08} coupled with technical testing estimates.

Using a refined method of decomposing Calder\'on-Zygmund operators as sums of dyadic shifts and an improvement of the Lacey-Petermichl-Reguera estimates, Hyt\"onen resolved the $A_2$ conjecture in 2012 in \cite{h12} and showed
\[
\| T  \|_{L^2(w) \rightarrow L^2(w)} \lesssim [w]_{A_2}
\] 
for all Calder\'on-Zygmund operators $T.$

We are interested the analogue of the $A_2$ conjecture in the setting of matrix weights. However, due to complications arising in the matrix case, the current literature is less developed. Still, the boundedness of Calder\'on-Zygmund operators is known. In 1997, Treil-Volberg showed
in \cite{vt97} that the Hilbert transform $H$ extends to a bounded operator on $L^2(W)$ if and only if $W$ is an $A_2$ matrix weight, i.e.~if and only if 
\[
\big [ W \big ] _{A_2} \equiv \sup_{I} \left \| \langle W \rangle_I^{\frac{1}{2}}
 \langle W^{-1} \rangle_I^{\frac{1}{2}} \right \|^2 < \infty,
\]
where $\| \cdot \|$ denotes the norm of the matrix acting on $\mathbb{C}^d$.
Soon after, Nazarov-Treil \cite{nt97} extended this result to general (classical) Calder\'on-Zygmund operators and in the interim, the study of operators on matrix-weighted spaces has received a great deal of attention. See \cite{cg01, gold03, IKP, lt07,  nptv02, vol97}.  However, the question of the sharp dependence on $[W]_{A_2}$ is still open and this seems to be a very difficult problem. In \cite{bpw14}, the two authors with S. Petermichl showed that  
\[ 
\| H  \|_{L^2(W) \rightarrow L^2(W)} \lesssim [W]_{A_2}^{\frac{3}{2}} \textnormal{log}\, [W]_{A_2},
\]
for all $A_2$ weights $W$, but this bound is unlikely to be sharp. 

Rather, a proof yielding a sharp estimate would likely follow, as in the scalar case, from the combination of (1) a sharp T1 theorem and (2) appropriate testing estimates.  The goal of this paper is to establish the T1 theorem and specifically, obtain matrix generalizations of the two-weight T1 theorems of Nazarov-Treil-Volberg from \cite{ntv08} about ``almost diagonal'' operators including Haar multipliers and dyadic shifts. These generalizations are interesting in their own right because they give two-weight results for all pairs of matrix $A_2$ weights, which is a new development. It seems possible that, as in the scalar case, these T1 theorems will prove a robust tool for studying the dependence of operator norms on the $A_2$ characteristic.   Before discussing the main results in more detail, we require several definitions.

\subsection{The Main Results}

Throughout the paper, $\mathcal{D}$ denotes the standard dyadic grid on $\mathbb{R}$ and $A \lesssim B$ means $A \le C(d) B$, where $C(d)$ is a (absolute) dimensional constant. For $I \in \mathcal{D}$, let $h_I$ be the standard Haar function defined by
\[ 
h_I \equiv  |I|^{-\frac{1}{2}} \left( \textbf{1}_{I_+} - \textbf{1}_{I_-} \right),\]
where $I_+$ is the right half of $I$ and $I_-$ is the left half of $I$. To the dyadic grid $\mathcal{D}$, associate the unique binary tree where each $I$ is connected to its two children $I_-$ and $I_+.$ Given that dyadic tree,  let $d_{\text{tree}}(I,J)$ denote the ``tree distance'' between $I$ and $J$, namely, the number of edges on the shortest path connecting  $I$ and $J$. The ``almost diagonal'' operators of interest possess a band structure defined as follows:

\begin{definition} A bounded operator $T$ on $L^2$ is a called a \emph{band operator with radius r} if  $T$ satisfies
\[ \left \langle T h_I e, h_J v \right \rangle_{L^2} = 0 \]
 for all intervals $I, J \in \mathcal{D}$ with $d_{\text{tree}}(I,J) >r$ and vectors $e, v \in \mathbb{C}^d.$
\end{definition}
Given a matrix weight $W$ and interval $I$ in $\mathcal{D}$, define the matrices: 
\[ W(I) \equiv \int_I W(x) \ dx \ \text{ and } \ \left \langle W \right \rangle_I \equiv \frac{1}{|I|} \int_I W(x) \ dx=\frac{W(I)}{\left\vert I\right\vert}. 
\]
In this paper, we will only consider weights $W$ with the property of being an $A_2$ weight, and without loss of generality, we can focus on the question of when a band operator $T$ extends to a bounded operator from $L^2(W^{-1})$ to $L^2(V)$ with norm $C$ for matrix weights $V, W.$ It is not hard to show that this occurs precisely when
\[  \left \|  M_{V^{\frac{1}{2}}} T M_{W^{\frac{1}{2}}}  \right \|_{L^2 \rightarrow L^2} = C. \]
The main results of this paper are then the following theorems.

\begin{theorem} \label{thm:Band} Let $W, V$ be matrix $A_2$ weights and let  $T$ be a band operator with radius $r$. Then $M_{V^{\frac{1}{2}}} T M_{W^{\frac{1}{2}}}$ extends to a bounded operator on $L^2$ if and only if 
\begin{eqnarray} \label{eqn:T}
\left \| T W \textbf{1}_I e \right \|_{L^2(V)} &\le A_1 \left \langle W(I) e, e \right \rangle_{\mathbb{C}^d}^{\frac{1}{2}} \\ \label{eqn:dual}
 \left \| T^* V \textbf{1}_I e \right \|_{L^2(W)} &\le A_2 \left \langle V(I) e, e \right \rangle_{\mathbb{C}^d}^{\frac{1}{2}}
\end{eqnarray}
 for all intervals $I \in \mathcal{D}$ and vectors $e \in \mathbb{C}^d$. Furthermore, 
 \[ \left \| M_{V^{\frac{1}{2}}} T M_{W^{\frac{1}{2}}} \right \|_{L^2 \rightarrow  L^2} \le 2^{2r} C(d) \left( A_1B(W) + A_2 B(V) \right), \]
 where $C(d)$ is a dimensional constant and $B(W)$ and $B(V)$ are constants depending on $W$ and $V$ from an application of the matrix Carleson Embedding Theorem.
 \end{theorem} 

The definitions of the constants $B(W)$ and $B(V)$ are given in Theorem \ref{thm:CET2}, the matrix Carleson Embedding Theorem used in this paper, and discussed further in Remark \ref{rem:CET}. As in \cite{ntv08}, the conditions of Theorem \ref{thm:Band} can be relaxed slightly to yield the following result:

\begin{theorem} \label{thm:Band2} Let $W, V$ be matrix $A_2$ weights and let  $T$ be a band operator with radius $r$. Then $M_{V^{\frac{1}{2}}} T M_{W^{\frac{1}{2}}}$ extends to a bounded operator on $L^2$ if and only if the following two conditions hold:
\begin{itemize} 
\item[$(i)$] For all intervals $I \in \mathcal{D}$ and vectors $e \in \mathbb{C}^d$,
\begin{eqnarray*} 
\left \| \textbf{1}_I T W \textbf{1}_I e \right \|_{L^2(V)} &\le A_1 \left \langle W(I) e, e \right \rangle_{\mathbb{C}^d}^{\frac{1}{2}} \\ 
 \left \| \textbf{1}_IT^* V \textbf{1}_I e \right \|_{L^2(W)} &\le A_2 \left \langle V(I) e, e \right \rangle_{\mathbb{C}^d}^{\frac{1}{2}}.
\end{eqnarray*}
 \item[$(ii)$] For all intervals $I, J \in \mathcal{D}$ satisfying $2^{-r} |I| \le |J| \le 2^r |I|$ and vectors $e, \nu \in \mathbb{C}^d$,
 \[ \left| \left \langle T_W \textbf{1}_I e, \textbf{1}_J \nu \right \rangle_{L^2(V)} \right| \le
 A_3 \left \langle W(I)e,e \right \rangle^{\frac{1}{2}}_{\mathbb{C}^d}
 \left \langle V(J) \nu,\nu \right \rangle^{\frac{1}{2}}_{\mathbb{C}^d}.
 \]
 \end{itemize}
  Furthermore, 
 \[ \left \| M_{V^{\frac{1}{2}}} T M_{W^{\frac{1}{2}}} \right \|_{L^2 \rightarrow  L^2} \le 2^{2r} C(d) \left( A_1B(W) + A_2 B(V) +A_3\right), \]
 where $C(d)$ is a dimensional constant and $B(W)$ and $B(V)$ are constants depending on $W$ and $V$ from an application of the matrix Carleson Embedding Theorem.
 \end{theorem} 

\begin{remark}  An observant reader, and expert in the area, will notice that Theorems \ref{thm:Band} and \ref{thm:Band2} are strictly weaker than the results of Nazarov-Treil-Volberg \cite{ntv08} in two respects.  First, our results are only proved for pairs $V,W$ of matrix $A_2$ weights and second, they introduce additional constants $B(V)$ and $B(W)$ in the norm estimates, which do not come from the testing conditions. 

However, it is worth pointing out that both of these shortcomings are the direct result of differences between the scalar Carleson Embedding Theorem and the current matrix Carleson Embedding Theorem. In the scalar case, the Carleson Embedding Theorem holds for \textit{all} weights and the embedding constant is an absolute multiple of the constant obtained from the testing condition. In the matrix case, the current Carleson Embedding Theorem, Theorem \ref{thm:CET2}, is only known for matrix $A_2$ weights and the embedding constant is the testing constant times an additional constant $B(W)$, depending upon the weight $W$.

A careful reading of our paper reveals that, if one can improve the underlying matrix Carleson Embedding Theorem in these two respects, then our arguments will give T1 theorems with sharp constants that hold for all pairs of matrix weights. It then seems likely that these results could be used as a tool to approach the matrix $A_2$ conjecture, at least in the setting of dyadic shifts and related operators.

Indeed, the authors recently learned that Amalia Culiuc and Sergei Treil have obtained an improved Carleson Embedding Theorem for arbitrary matrix weights in the more general non-homogeneous setting.  The two authors with Culiuc and Treil are currently investigating the behavior of well-localized operators in this more general setting. 

It is also worth observing that related and interesting results are obtained by R. Kerr in \cite{rk11, rk14}. He shows that band operators on $L^2$  will be bounded from $L^2(W)$ to $L^2(V)$ if the matrix weights $V$ and $W$ are both in the matrix analogue of $A_{\infty}$ (denoted $A_{2,0}$) and satisfy a joint $A_2$ condition.
\end{remark}

\begin{remark} \label{rem:L1loc} If the entries of $W,V$ are not locally square-integrable, i.e.~ not in $L^2_{loc}(\mathbb{R})$,  one needs to be a little careful about interpreting the expressions on the left-hand sides of \eqref{eqn:T} and \eqref{eqn:dual} and the analogous expressions in Theorem \ref{thm:Band2}. This technicality can be handled in a way similar to that found in \cite{ntv08}. Indeed, observe that if $W, W'$ are matrix weights satisfying $W' \le W$, then 
\[ \left \| M_{W'^{\frac{1}{2}}} T^* M_{V^{\frac{1}{2}}} \right \|_{L^2 \rightarrow L^2} \le \left \| M_{W^{\frac{1}{2}} }T^* M_{V^{\frac{1}{2}}} \right \|_{L^2 \rightarrow L^2} \]
and taking adjoints gives
\[ \left \| M_{V^{\frac{1}{2}}} T M_{W'^{\frac{1}{2}}} \right \|_{L^2 \rightarrow L^2} \le \left \| M_{V^{\frac{1}{2}}} T M_{W^{\frac{1}{2}}} \right \|_{L^2 \rightarrow L^2}. \]
Now, to interpret the first necessary condition appropriately, let $\{W_n\}$ be a sequence of matrix weights with entries in $L^2_{loc}(\mathbb{R})$ increasing to $W$. Then, the boundedness of $M_{V^{\frac{1}{2}}} T M_{W^{\frac{1}{2}}}$ implies that 
\[ \left \|  T W_n \textbf{1}_I e \right \|_{L^2(V)} \le C < \infty \]
for some constant $C$ uniformly in $n$. It is not difficult to show that this implies $\left\{  M_{V^{\frac{1}{2}}} T W_n \textbf{1}_I e\right\}$ has a limit in $L^2$, which is independent of the sequence $\{W_n\}$ chosen. So, there is no ambiguity in calling this limit function $V^{\frac{1}{2}} T W \textbf{1}_Ie$ and interpreting the lefthand side of \eqref{eqn:T} as its $L^2$ norm. The dual expressions are interpreted analogously. We can similarly interpret the term in $(ii)$ from Theorem \ref{thm:Band2} as the inner product between $V^{\frac{1}{2}} T W \textbf{1}_Ie$ and $V^{\frac{1}{2}} \textbf{1}_J \nu$ in $L^2.$ 

To interpret the sufficient condition, fix any sequences $\{W_n\}$ and $\{V_n\}$ in $L^2_{loc}(\mathbb{R})$ increasing to $W$ and $V$ respectively.  Conditions \eqref{eqn:T} and \eqref{eqn:dual} can be interpreted as the estimates
\[
\begin{aligned}
\left \| T W_n \textbf{1}_I e \right \|_{L^2(V_n)} &\le A_1 \left \langle W_n(I) e, e \right \rangle_{\mathbb{C}^d}^{\frac{1}{2}} \\ 
 \left \| T^* V_n \textbf{1}_I e \right \|_{L^2(W_n)} &\le A_2 \left \langle V_n(I) e, e \right \rangle_{\mathbb{C}^d}^{\frac{1}{2}},
\end{aligned}
\]
which are uniform in $n$, $e$, and $I$. Then Theorem \ref{thm:Band} gives the bound for $\left\| M_{V_n^{\frac{1}{2}}} T M_{W_n^{\frac{1}{2}}} \right\|_{L^2\to L^2}$ which implies the desired bound for 
 $\left\| M_{V^{\frac{1}{2}}} T M_{W^{\frac{1}{2}}} \right\|_{L^2\to L^2}$. The analogous interpretations of the expressions in Theorem \ref{thm:Band2} should also be clear.
\end{remark}

\subsection{Summary and Outline of the Paper}

The remainder of the paper consists of the proofs of Theorems \ref{thm:Band} and \ref{thm:Band2}. To outline the proof technique, assume that $W$, $V$ are matrix $A_2$ weights. It is not hard to show that $M_{V^{\frac{1}{2}}} T M_{W^{\frac{1}{2}}}: L^2 \rightarrow L^2$ is bounded with operator norm $C$ if and only if the operator
\[ T_W \equiv TM_{W}: L^2(W) \rightarrow L^2(V) \ \text{ satisfies } \   \|T_W\|_{L^2(W) \rightarrow L^2(V)} =C.   \] 
Because $T$ is a band operator, $T_W$ will have a particularly nice structure. Following the language and proof strategy of Nazarov-Treil-Volberg \cite{ntv08}, we will show $T_W$ is well-localized. Section \ref{sec:WellLoc} contains the details of well-localized operators, their connections to band operators, and the analogues of Theorems \ref{thm:Band} and \ref{thm:Band2} for well-localized operators. We call these results Theorems \ref{thm:WellLoc} and \ref{thm:WellLoc2}. These theorems will immediately imply our main results:~Theorems \ref{thm:Band} and \ref{thm:Band2}.

In Sections  \ref{sec:HaarBasis} and  \ref{sec:MCET}, the paper develops the tools need to prove Theorems \ref{thm:WellLoc} and \ref{thm:WellLoc2}.  In Section  \ref{sec:HaarBasis}, we define and outline the properties of a system of Haar functions adapted to a general matrix weight $W$. This system appears to be new in the context of matrix weights. We also require a matrix Carleson Embedding Theorem. We use the ideas of Treil-Volberg \cite{vt97} and Isralowitz-Kwon-Pott \cite{IKP} to obtain such a theorem with the best known constant. Details are given in Section \ref{sec:MCET}. 

Section \ref{sec:proof} contains the proofs of Theorems \ref{thm:WellLoc} and \ref{thm:WellLoc2}. The well-localized structure of $T_W$ makes $T_W$ amenable to separate analyses of its diagonal part and upper and lower triangular parts, which behave like nice paraproducts.  We compute the norm by duality and as part of the argument, decompose the functions in question relative to weighted Haar bases adapted to $W$ and $V$ respectively.  To control the upper and lower triangular pieces, we define associated paraproducts and show they are bounded using the testing hypothesis and matrix Carleson Embedding Theorem. We bound the diagonal pieces using the well-localized structure of $T_W$ coupled with properties of the system of Haar functions and the given testing conditions.

\section{Weighted Haar Basis} \label{sec:HaarBasis}
Let $W$ be a matrix weight, and let $\| \cdot \|$ denote the operator norm of a matrix on $\mathbb{C}^d$.  In this section, we construct a set of disbalanced Haar functions adapted to $W$, which we denote $H_W$.  First, fix $J \in \mathcal{D}$ and  let $v^1_J, \dots, v^d_J$ be a set of orthonormal 
eigenvectors of the positive matrix:
\begin{equation}
\begin{aligned}
\label{posmat}
W(J_-)W(J_+)^{-1}W(J_-)   + W(J_-) &= W(J_-)W(J_+)^{-1}W(J_-)   + W(J_+)W(J_+)^{-1}W(J_-) \\
&  =  W(J)W(J_+)^{-1}W(J_-). 
\end{aligned}
\end{equation}
Furthermore, for $1 \le j \le d$, define the constant
\[ w_J^j  \equiv  \left \|  \left( W(J)W(J_+)^{-1}W(J_-) \right)^{\frac{1}{2}} v^j_J \right \|. \]
Since the matrix \eqref{posmat} is positive and $v^j_J$ is a normalized eigenvector, it follows that:
\[ ( w_J^j)^{-1} v^j_J =  \left( W(J)W(J_+)^{-1}W(J_-) \right)^{-\frac{1}{2}} v^j_J \qquad \forall 1 \le j \le d. \]
\begin{definition} \label{def:weightHaar} For each $J \in \mathcal{D}$, define the \emph{vector-valued Haar functions on $J$ adapted to $W$} as follows:
\begin{equation} \label{eqn:haarfunctions}
h^{W,j}_J \equiv ( w_J^j)^{-1} \left( W(J_+)^{-1}W(J_-)v_J^j \textbf{1}_{J_+} - v_J^j \textbf{1}_{J_-} \right) \qquad \forall 1\le j \le d. \end{equation}
If the constant function $\textbf{1}_{[0, \infty)} e$ is in $L^2(W)$ for any nonzero $e$ in $\mathbb{C}^d$, let $\{ e_1, \dots, e_{p_1} \}$ be an orthonormal basis of the subspace of $\mathbb{C}^d$ satisfying  $\textbf{1}_{[0, \infty)} e \in L^2(W).$ Define 
\[ h_{1}^{W,i} \equiv c^i_1 \textbf{1}_{[0, \infty)} e_i \qquad \text{ for } i =1, \dots, p_1, \]
where $c^i_1$ is chosen so that $\| h_{1}^{W,i} \|_{L^2(W)} =1.$ Define the functions
\[ h_{2}^{W,i} \equiv c^i_2 \textbf{1}_{(-\infty, 0]} \nu_i \qquad \text{ for } i =1, \dots, p_2, \]
where 
$\{ \nu_1, \dots, \nu_{p_2} \}$ is an orthonormal basis of the subspace of $\mathbb{C}^d$ satisfying  $\textbf{1}_{(-\infty, 0]} \nu \in L^2(W),$ in an analogous way. Define $H_W$, the \emph{system of Haar functions adapted to $W,$} by: 
\[H_W \equiv \left \{ h^{W,j}_J \right\}  \cup  \left \{ h_{k}^{W,i} \right\}.\]
One should notice that if the constant functions $\textbf{1}_{ [0,\infty)} e$ and $\textbf{1}_{(-\infty, 0]} e$ are not in $L^2(V)$ for all $e \in \mathbb{C}^d$, then $H_W = \left \{ h^{W,j}_J \right\}$.
\end{definition}

We now show that $H_W$ is an orthonormal basis of $L^2(W).$

\begin{lemma} The system $H_W$ is an orthonormal system in $L^2(W).$ \end{lemma}

\begin{proof} We first prove that the system $\left\{ h^{W,j}_J\right\}$ is orthogonal. Fix $h^{W,j}_J$ and $h^{W,i}_I.$ First, assume $I \ne J$. Then, one interval must be strictly contained in the other because otherwise, the inner product trivially vanishes by support conditions. Without loss of generality, assume $I \subsetneq J$. This implies that $h^{W,j}_J$ equals a constant vector on $I$, which we  will denote by $e$. Then
\[
\begin{aligned}
\left \langle h^{W,i}_I,  h^{W,j}_J \right \rangle_{{L^2(W)} }
&=\int_{I} \left \langle W(x) h^{W,i}_I,   e \right \rangle_{\mathbb{C}^d} dx \\
&= \int_I   ( w_I^i)^{-1} \left \langle W(x) \left( W(I_+)^{-1}W(I_-)v_I^i \textbf{1}_{I_+} - v_I^i \textbf{1}_{I_-} \right), e \right \rangle_{\mathbb{C}^d} dx \\
& = ( w_I^i)^{-1} \left \langle W(I_+)  W(I_+)^{-1}W(I_-)v_I^i, e \right \rangle_{\mathbb{C}^d} -   ( w_I^i)^{-1}
\left \langle W(I_-) v_I^i, e \right \rangle_{\mathbb{C}^d}\\
 &=0.
\end{aligned}
 \]
One should notice that the definition of $e$ played no role; in fact, the above arguments show that each $h^{W,j}_J$ has mean zero with respect to $W$. Now assume $I=J$ and $i \ne j$. Observe that:
\[
\begin{aligned}
&\left \langle  h^{W,i}_J,  h^{W,j}_J \right \rangle_{L^2(W)}
=\int_{J} \left \langle W(x) h^{W,i}_J,   h^{W,j}_J \right \rangle_{\mathbb{C}^d} dx \\
& = ( w_J^j)^{-1}  ( w_J^i)^{-1}  \int_J \left \langle W(x) \left( W(J_+)^{-1}W(J_-)v_J^i \textbf{1}_{J_+} - v_J^i \textbf{1}_{J_-} \right), W(J_+)^{-1}W(J_-)v_J^j \textbf{1}_{J_+} - v_J^j \textbf{1}_{J_-}
\right \rangle_{\mathbb{C}^d} dx \\
& =  (w_J^j)^{-1} ( w_J^i)^{-1} \left(
\left \langle W(J_+) W(J_+)^{-1}W(J_-)v_J^i, W(J_+)^{-1}W(J_-)v_J^j \right \rangle_{\mathbb{C}^d} + \left \langle W(J_-) v_J^i,  v_J^j \right \rangle_{\mathbb{C}^d}  \right)\\
&= (w_J^j)^{-1} ( w_J^i)^{-1} \left \langle \left(W(J_-)W(J_+)^{-1}W(J_-)   + W(J_-) \right)v_J^i, v_J^j \right \rangle_{\mathbb{C}^d} \\
&=0,
\end{aligned}
\]
since $v^i_J$ and $v^j_J$ are orthonormal eigenvectors of $W(J_-)W(J_+)^{-1}W(J_-)   + W(J_-)$. Since each $h^{W,j}_J$ has mean zero with respect to $W$ and since each  $h^{W,j}_J$ is either supported in $(-\infty, 0]$ or $[0, \infty)$, it is clear that
\[  \left \langle h^{W,j}_J, h_{k}^{W,i} \right \rangle_{L^2(W)} =0 \qquad \forall J \in \mathcal{D} \]
and for all indices $i, j, k.$  By construction, it is also clear that $\left\{ h_{k}^{W,j} \right\}$ is an orthonormal set in $L^2(W).$ Finally, to see that $\left\{ h^{W,j}_J  \right\}$ is normalized, fix $h^{W,j}_J$ and observe that
\[
\begin{aligned}
&\left \langle  h^{W,j}_J,  h^{W,j}_J \right \rangle_{{L^2(W)} }
=(w_J^j)^{-2} \left \langle \left( W(J_-)W(J_+)^{-1}W(J_-)  + W(J_-) \right) v_J^j, v_J^j \right \rangle_{\mathbb{C}^d} \\
& =\left \langle \left( W(J_-)W(J_+)^{-1}W(J_-)  + W(J_-) \right) \left( W(J_-)W(J_+)^{-1}W(J_-)  + W(J_-) \right)^{-1} v_J^j, v_J^j \right \rangle_{\mathbb{C}^d}  \\
&=1,
\end{aligned}
\]
using the properties of $v_J^j$ and the definition of $w_J^j$. This completes the proof.
\end{proof}

\begin{lemma} The orthonormal system $H_W$ is complete in $L^2(W).$ \end{lemma}

\begin{proof} 
Fix $f$ in $L^2(W)$, and assume $f$ is orthogonal to every function in $H_W$. Specifically, $f$ is orthogonal to the set $\left\{ h^{W,j}_J\right\}.$ Then, for each $J \in \mathcal{D}$ and $j=1, \dots, d,$
\[ 0 =  \left \langle  f,  h^{W,j}_J \right \rangle_{{L^2(W)} }. \]
Multiplying by a constant gives:
\[ 
\begin{aligned}
0 &=  |J_-|^{-1} \left \langle  W(J_+)^{-1}W(J_-)v_J^j \textbf{1}_{J_+} - v_J^j \textbf{1}_{J_-}, f  \right \rangle_{{L^2(W)} } \\
&= |J_-|^{-1}   \int_J \left \langle  W(J_+)^{-1}W(J_-)v_J^j \textbf{1}_{J_+} - v_J^j \textbf{1}_{J_-}, W(x) f(x) \right \rangle_{\mathbb{C}^d} dx \\
& = \left \langle  W(J_+)^{-1}W(J_-)v_J^j , \left \langle W f \right \rangle_{J_+} \right \rangle_{\mathbb{C}^d}  - \left \langle v_J^j ,  \left \langle W f \right \rangle_{J_-} \right \rangle_{\mathbb{C}^d} \\
&= \left \langle v_J^j, W(J_-)W(J_+)^{-1} \left \langle W f \right \rangle_{J_+} -   \left \langle W f \right \rangle_{J_-} \right \rangle_{\mathbb{C}^d}.
\end{aligned}
\]
Since this holds for each $j$ and $v_J^1, \dots, v^d_J$ is an orthonormal basis of $\mathbb{C}^d$, we can conclude that 
\begin{equation} \label{eqn:averages}   \left \langle W f \right \rangle_{J_-} = W(J_-)W(J_+)^{-1} \left \langle W f \right \rangle_{J_+}. \end{equation}
Adding $\left \langle Wf \right \rangle_{J_+}$ to both sides gives
\[ 2\left \langle W f \right \rangle_{J} = W(J_-)W(J_+)^{-1} \left \langle W f \right \rangle_{J_+} +\left \langle Wf \right \rangle_{J_+}= \left( W(J_-)W(J_+)^{-1}  + W(J_+)W(J_+)^{-1} \right) \left \langle W f \right \rangle_{J_+}.  \]
Rearranging by factoring out $W(J_+)^{-1}$ on the right from the term in parentheses and using the definitions gives
\[  \left \langle W \right \rangle_J^{-1} \left \langle Wf  \right \rangle_J = \left \langle W  \right \rangle_{J_+}^{-1} \left \langle Wf  \right \rangle_{J_+}.  \]
Solving \eqref{eqn:averages} for $\left \langle W f \right \rangle_{J_+}$ and using analogous arguments, one can show:
\[  \left \langle W \right \rangle_J^{-1} \left \langle Wf  \right \rangle_J = \left \langle W  \right \rangle_{J_-}^{-1} \left \langle Wf  \right \rangle_{J_-}.  \]
Now fix any $x,y \in (0, \infty)$ and choose some dyadic interval $J_0$ so that $x,y \in J_0.$ Define two sequence of dyadic intervals:
\[
\begin{aligned}
J_0 &= I_0 \supsetneq I_1 \supsetneq I_2 \dots \supsetneq I_i \supsetneq I_{i+1} \dots \\
J_0 &= K_0 \supsetneq K_1 \supsetneq K_2 \dots \supsetneq K_k \supsetneq K_{k+1} \dots
\end{aligned}
\]
such that each $I_i$ is a parent of $I_{i+1}$ and $x \in I_i$ for all $i$ and similarly, each $K_k$ is a parent of $K_{k+1}$ and $y$ is in each $K_k$. Our previous arguments imply that
\[ \left \langle W \right \rangle_{I_i}^{-1} \left \langle Wf  \right \rangle_{I_i}  = \left \langle W \right \rangle_{J_0}^{-1} \left \langle Wf  \right \rangle_{J_0}  = \left \langle W \right \rangle_{K_k}^{-1} \left \langle Wf  \right \rangle_{K_k} \quad \forall i,k \in \mathbb{N}. \] 
Now we can use the Lebesgue Differentiation Theorem to conclude that 
\[ W(x)^{-1} W(x) f(x) = W(y)^{-1} W(y) f(y) \]
for almost every $x,y$ in $(0, \infty)$ and so $f(x) = f(y)$ for almost every $x,y$ in $[0, \infty)$. Analogous arguments imply $f$ must be constant on $(- \infty, 0]$. But, by assumption, $f$ is also orthogonal to the set $\{ h^{W,i}_k\}$, which implies $f$ is orthogonal to all of the nonzero constant functions supported on  $[0, \infty)$ or $(- \infty, 0]$ in $L^2(W).$ Thus, we can conclude $f \equiv 0.$ \end{proof}

We require one additional fact about the weighted Haar system:
\begin{lemma} \label{lem:haarbound} The orthonormal system $H_W$ satisfies
\[
\begin{aligned}
\left \| W(J_-)^{\frac{1}{2}} h_J^{W,j}(J_-) \right \|_{\mathbb{C}^d} & \le C(d) \\
\left \| W(J_+)^{\frac{1}{2}} h_J^{W,j}(J_+) \right \|_{\mathbb{C}^d} & \le C(d) \\
\end{aligned}
\] 
for all $J \in \mathcal{D}$ and $1 \le j \le d,$ where $h_J^{W,j}(J_\pm)$ is the constant value $h_{J}^{W,j}$ takes on $J_\pm$.
\end{lemma}
\begin{proof} We only prove the first inequality as the second is proved similarly. First, recall that 
$W(J)W(J_+)^{-1}W(J_-)$ is a positive matrix and hence, $W(J_-)^{-1}W(J_+)W(J)^{-1}$ is positive as well. Now, observe that
\[
\begin{aligned}
  \left \| W(J_-)^{\frac{1}{2}} h_J^{W,j} \left( J_- \right ) \right \|_{\mathbb{C}^d}^2 
& \le \left \| W(J_-)^{\frac{1}{2}} \left( W(J) W(J_+)^{-1} W(J_-) \right)^{-\frac{1}{2}} \right \|^2 \\
& =  \left \| W(J_-)^{\frac{1}{2}} W(J_-) ^{-1} W(J_+) W(J)^{-1} W(J_-)^{\frac{1}{2}} \right \| \\
& \le C(d)  \ \text{Tr}\left(W(J_-)^{\frac{1}{2}} W(J_-) ^{-1} W(J_+) W(J)^{-1} W(J_-)^{\frac{1}{2}} \right) \\
&= C(d) \  \text{Tr}\left( W(J)^{-\frac{1}{2}}  W(J_+) W(J)^{-\frac{1}{2}}  \right)  \\
& \le C(d) \left \| W(J)^{-\frac{1}{2}}  W(J_+) W(J)^{-\frac{1}{2}}  \right \| \\
& \le C(d) \left \| W(J)^{-\frac{1}{2}}  W(J) W(J)^{-\frac{1}{2}}  \right \| \\
& = C(d), 
\end{aligned}
\]
where we used the fact that trace and operator norm are equivalent (up to a dimensional constant) for positive matrices. This completes the proof. 
\end{proof}

\begin{remark} \label{rem:Expand} In the proofs of Theorems \ref{thm:WellLoc} and \ref{thm:WellLoc2}, we will expand functions in $L^2(W)$ with respect to the basis $H_W.$ Specifically, if $f \in L^2(W)$, we can expand $f$ as 
\[ f = \sum_{\substack{ J \in \mathcal{D} \\ 1\le j \le d}} \left \langle f, h^{W,j}_J \right \rangle_{L^2(W)} h^{W,j}_J + \sum_{\substack{1 \le k \le 2 \\ 1\le j \le p_k}} \left \langle f, h^{W,j}_k \right \rangle_{L^2(W)}  h^{W,j}_k.\]
This means that for $K \in \mathcal{D}$, we can express the weighted average of $f$ on $K$ as
\[
\begin{aligned}
 \left \langle W \right \rangle^{-1}_K  \left \langle W f \right \rangle _K &=  
\sum_{\substack{ J \in \mathcal{D} \\ 1\le j \le d}} \left \langle f, h^{W,j}_J \right \rangle_{L^2(W)}   \left \langle W \right \rangle^{-1}_K  \left \langle W h^{W,j}_J \right \rangle_K \\
& \ \ \ \ + \sum_{\substack{1 \le k \le 2 \\ 1\le j \le p_k}} \left \langle f, h^{W,j}_k  \right \rangle_{L^2(W)}   \left \langle W \right \rangle^{-1}_K   \left \langle W h^{W,j}_k \right \rangle_{K}  \\
&=  
\sum_{\substack{J: K \subsetneq J \\ 1\le j \le d}} \left \langle f, h^{W,j}_J \right \rangle_{L^2(W)}    h^{W,j}_J(K) + \sum_{\substack{1 \le k \le 2 \\ 1\le j \le p_k}} \left \langle f, h^{W,j}_k \right \rangle_{L^2(W)}  h^{W,j}_k(K),
\end{aligned}
\]
where $h^{W,j}_J(K)$ is the constant value that $h^{W,j}_J$ takes on $K$ and $h^{W,j}_k(K)$ is the constant value that $h^{W,j}_k$ takes on $K$. Now, assume $f$ is compactly supported, so that we can find two dyadic intervals  $I_1 \subset [0, \infty)$ and $I_2 \subset (-\infty, 0]$ such that $\text{supp}(f) \subseteq I_1 \cup I_2.$ For $I \in \mathcal{D}$, define the the weighted expectation of $f$ on $I$ by  
\[E^W_{I} f \equiv \left \langle W \right \rangle^{-1}_{I}  \left \langle W f \right \rangle _{I}\textbf{1}_{I}. \]
Then, we can write $f$ as 
\begin{eqnarray}
\nonumber f &=& \sum_{\substack{J \in \mathcal{D}\\ 1\le j \le d}} \left \langle f, h^{W,j}_J \right \rangle_{L^2(W)}  h^{W,j}_J + \sum_{\substack{1 \le k \le 2 \\ 1\le j \le p_k}} \left \langle f, h^{W,j}_k \right \rangle_{L^2(W)}  h^{W,j}_k \\
\nonumber &=&  \sum_{\substack{J:J \subseteq I_1 \cup I_2 \\ 1\le j \le d}} \left \langle f, h^{W,j}_J \right \rangle_{L^2(W)}  h^{W,j}_J + \sum_{1 \le \ell \le 2} \left \langle W \right \rangle^{-1}_{I_{\ell}}  \left \langle W f \right \rangle _{I_{\ell}}\textbf{1}_{I_{\ell}} \\
 &=&  \sum_{\substack{J:J \subseteq I_1 \cup I_2 \\ 1\le j \le d}} \left \langle f, h^{W,j}_J \right \rangle_{L^2(W)}  h^{W,j}_J + \sum_{1 \le \ell \le 2} E^W_{I_{\ell}} f. \label{eqn:sum}
 \end{eqnarray}

\end{remark}

\section{Matrix Carleson Embedding Theorem} \label{sec:MCET}

Let $W$ be a matrix weight such that for all positive semi-definite matrices $A$ and intervals $J \in\mathcal{D}$, there is a uniform constant $C$ satisfying
\begin{equation} \label{eqn:reverse} \frac{1}{|J|} \int_J \| A W(x) A \| dx \le C \left( \frac{1}{|J|} \int_J \| A W(x) A \|^{\frac{1}{2}} dx \right)^2. \end{equation}
Define $[W]_{R_2}$ to be the smallest such constant $C$. Treil-Volberg's arguments in Lemma 3.5 and Lemma 3.6 in \cite{vt97} show that, if $W$ is an $A_2$ matrix weight, then
\begin{equation} \label{eqn:RHconstant} 
[W]_{R_2} \le C(d) [W]_{A_2}.
\end{equation}
In Theorem 6.1 in \cite{vt97}, Treil-Volberg prove an embedding theorem for a specific sequence of positive semi-definite matrices. Their arguments generalize easily to arbitrary sequences of matrices, yielding the following matrix Carleson Embedding Theorem:

\begin{theorem} \label{thm:CET1} Let $W$ be a matrix weight satisfying \eqref{eqn:reverse}  and let $\left\{A_I\right\}_{I\in\mathcal{D}}$ be a sequence of positive semi-definite $d\times d$ matrices.  Then 
\[
\sum_{I\in\mathcal{D}} \left\langle A_I \left\langle f\right\rangle_I, \left\langle f\right\rangle_I\right\rangle_{\mathbb{C}^d} \le C_1 \left\Vert f\right\Vert_{L^2(W^{-1})}^2 \  \text{ if } \ 
\frac{1}{|J|} \sum_{I:I \subseteq J} \left \| \left \langle W \right \rangle^{\frac{1}{2}}_I A_I  \left \langle W \right \rangle^{\frac{1}{2}}_I \right \| \le C_2 \ \ \forall J \in \mathcal{D},
\]
where $C_1 = C_2 C(d) [W]_{R_2}$ and $C(d)$ is a dimensional constant.
\end{theorem}

It should be noted that in \cite{IKP}, Isralowitz-Kwon-Pott obtained a  more general version of Theorem \ref{thm:CET1}, which holds for all $A_p$ matrix weights.

\begin{remark} Treil-Volberg's arguments in \cite{vt97} actually establish a seemingly stronger result. Namely, they show that if $\{ B_I \}_{I\in\mathcal{D}}$ is a sequence of positive semi-definite matrices, then 
\begin{equation} \label{eqn:TV} \sum_{I \in \mathcal{D}} \left \| \left \langle W \right \rangle_I^{-\frac{1}{2}} B_I\left \langle W \right \rangle_I^{-\frac{1}{2}} \right \|  \left \| \left \langle W \right \rangle_I^{-\frac{1}{2}}  \left \langle W^{\frac{1}{2}} g \right \rangle_I \right \|^2_{\mathbb{C}^d} \le C_1 \|  g\|_{L^2}^2 \ \text{ if } \
\frac{1}{|J|}\sum_{I:I \subseteq J} \left \| \left \langle W \right \rangle_I^{-\frac{1}{2}} B_I \left \langle W \right \rangle_I^{-\frac{1}{2}}  \right \| \le C_2,\end{equation}
for all $J \in \mathcal{D}$.
To recover Theorem \ref{thm:CET1} from \eqref{eqn:TV}, note that 
\[ 
\begin{aligned} 
\sum_{I \in \mathcal{D}}  \left\langle  \left \langle W \right \rangle_I^{-1} B_I\left \langle W \right \rangle_I^{-1}  \left \langle W^{\frac{1}{2}} g \right \rangle_I ,  \left \langle W^{\frac{1}{2}} g \right \rangle_I  \right \rangle_{\mathbb{C}^d} 
&\le  \sum_{I \in \mathcal{D}} \left \| \left \langle W \right \rangle_I^{-\frac{1}{2}} B_I\left \langle W \right \rangle_I^{-\frac{1}{2}} \right \|  \left \| \left \langle W \right \rangle_I^{-\frac{1}{2}}  \left \langle W^{\frac{1}{2}} g \right \rangle_I \right \|^2_{\mathbb{C}^d}.
\end{aligned}
\]
If one is given $\left\{A_I\right\}_{I\in\mathcal{D}}$ and $f \in L^2(W^{-1})$, then pairing the above inequality with \eqref{eqn:TV} using  $B_I \equiv \langle W \rangle_I A_I \langle W \rangle_I$ and $g \equiv W^{-\frac{1}{2}}f$ gives the inequalities in Theorem \ref{thm:CET1}.\end{remark}

Equation \eqref{eqn:TV} is proved via arguments similar to those used in \cite{nn86} to establish the standard Carleson Embedding Theorem. Specifically, Treil-Volberg define an associated embedding operator and show it is bounded using the Senichkin-Vinogradov Test:

\begin{theorem}[Senichkin-Vinogradov Test] \label{thm:SV} Let $\mathcal{Z}$ be a measure space, and let $k$ be a locally summable, nonnegative, measurable function on $\mathcal{Z} \times \mathcal{Z}$. If 
\[ \int_{\mathcal{Z}} k(s,t)k(s,x) \ ds \le C \left[ k(x,t) + k(t,x) \right] \quad a.e. \text{ on } \mathcal{Z},\]
then for all nonnegative $g \in L^2(\mathcal{Z})$,
\[ \int_{\mathcal{Z}} \int_{\mathcal{Z}} k(s,t)g(s) g(t) \ ds dt \le 2C \| g \|^2_{L^2(\mathcal{Z})}. 
\]
\end{theorem}

For the ease of the reader, we sketch the proof of \eqref{eqn:TV}. We focus on the first half of the proof, as the second half is given in detail in \cite{vt97}. 

\begin{proof} First define $\mu_I \equiv  \left \| \left \langle W \right \rangle_I^{-\frac{1}{2}} B_I\left \langle W \right \rangle_I^{-\frac{1}{2}} \right \| .$ Then, by assumption, $\{\mu_I \}_{I\in\mathcal{D}}$ is a scalar Carleson sequence with testing constant $C_2.$ Define the embedding operator $\mathcal{J}: L^2 \rightarrow \ell^2( \{ \mu_I \}, \mathbb{C}^d)$ by 
\[ \mathcal{J} f = \left\{ \left \langle W \right \rangle_I^{-\frac{1}{2}} \left \langle W^{\frac{1}{2}} f \right \rangle_I \right \}_{I \in \mathcal{D}}\]
and observe that \eqref{eqn:TV} is equivalent to $\mathcal{J}$ having operator norm bounded by $\sqrt{C_1}.$ To prove the norm bound, one shows that the formal adjoint $\mathcal{J}^*: \ell^2(\{\mu_I\}, \mathbb{C}^d) \rightarrow L^2$ defined by 
\[\mathcal{J}^* \{ \alpha_I \} \equiv \sum_{I \in \mathcal{D}} \frac{\mu_I}{|I|} \textbf{1}_I W^{\frac{1}{2}} \left \langle W \right \rangle_I^{-\frac{1}{2}} \alpha_I \qquad \forall \ \{\alpha_I \} \in  \ell^2 \left(\{\mu_I\}, \mathbb{C}^d\right)  \]
has the desired norm bound. First observe that
\[ \mathcal{J} \mathcal{J}^* \{\alpha_I \} = \left\{ \left \langle W \right \rangle_J^{-\frac{1}{2}} \sum_{I \in \mathcal{D}} \frac{\mu_I}{|I|} \left \langle W \textbf{1}_I \right \rangle_J \left \langle W \right \rangle_I^{-\frac{1}{2}} \alpha_I \right \}_{J \in \mathcal{D}}. \]
One can use this to immediately show that for any $\{ \alpha_I \}$ in $\ell^2( \{\mu_I\}, \mathbb{C}^d)$, 
\[
\begin{aligned}
\left \| \mathcal{J}^* \{ \alpha_I \} \right \|_{L^2}^2 \l &= \left \langle \mathcal{J} \mathcal{J}^* \{ \alpha_I \},  \{ \alpha_I \} \right \rangle_{\ell^2(\{\mu_I \}, \mathbb{C}^d)} \\ 
&= \sum_{J \in \mathcal{D}} \sum_{I: I \subseteq J} \frac{ \mu_I \mu_J}{|J|} \left \langle \left \langle W \right \rangle^{-\frac{1}{2}}_J \left \langle W \right \rangle^{\frac{1}{2}}_I \alpha_I, \alpha_J \right \rangle_{\mathbb{C}^d}  + \sum_{I \in \mathcal{D}} \sum_{J: J \subsetneq I} \frac{ \mu_I \mu_J}{|I|} \left \langle \left \langle W \right \rangle^{\frac{1}{2}}_J \left \langle W \right \rangle^{-\frac{1}{2}}_I \alpha_I, \alpha_J \right \rangle_{\mathbb{C}^d}.
\end{aligned}
\]
Now, for $K,L \in \mathcal{D}$, define $T_{LK}$ by 
\[T_{LK} \equiv \frac{1}{|L|} \left \| \left \langle W \right \rangle^{\frac{1}{2}}_K \left \langle W \right \rangle^{-\frac{1}{2}}_L \right \|
= \frac{1}{|L|} \left \| \left \langle W \right \rangle^{-\frac{1}{2}}_L  \left \langle W \right \rangle^{\frac{1}{2}}_K \right \|
\]
 if $K \subseteq L$ and $T_{KL} =0$ otherwise. By symmetry in the sums, it is easy to show that 
\begin{equation} \label{eqn:TVbound} 
\left \| \mathcal{J}^* \{ \alpha_I \} \right \|^2_{L^2} \l  \le 2 \sum_{J \in \mathcal{D}} \sum_{I: I \subseteq J}  \mu_I \mu_J T_{JI} \|\alpha_I \|_{\mathbb{C}^d} \| \alpha_J \|_{\mathbb{C}^d}.
\end{equation}
Thus, the result will be proved if one can show that the righthand side of \eqref{eqn:TVbound} is bounded by $C_1 \|\{ \alpha_I\} \|^2_{\ell^2( \{ \mu_I \}, \mathbb{C}^d)}.$ This is where one uses the Senichkin-Vinogradov Test. 
Let $\mathcal{Z}$ be $\mathcal{D}$, the set of dyadic intervals, with point mass $\mu_I$ on each  interval $I$. Then,  $L^2(\mathcal{Z})$ is equivalent to  $\ell^2(\{\mu_I\}, \mathbb{C}).$ Indeed, $\{\beta_I\}\in \ell^2(\{\mu_I\}, \mathbb{C})$ if and only if the function $\beta$ defined by $\beta(I) = \beta_I$ is in $L^2(\mathcal{Z})$. Moreover, 
\[ \| \{\beta_I\} \|_{\ell^2(\mu_I, \mathbb{C})} = \| \beta \|_{L^2(\mathcal{Z})},\]
so we can treat these as the same objects. Now, define the nonnegative function $k: \mathcal{Z} \times \mathcal{Z} \rightarrow \mathbb{R}^+$ by 
\[ k(K, L) \equiv \sum_{J \in \mathcal{D}} \sum_{I: I \subseteq J} T_{JI} \delta_I(K) \delta_J(L), \]
where $\delta_I(K)=1$ if $K=I$ and zero otherwise. Fix a sequence  $\{\alpha_I\}\in \ell^2( \{\mu_I \}, \mathbb{C}^d)$. Then the sequence $\{a_I\}$ defined by $a_I \equiv \| \alpha_I \|_{\mathbb{C}^d}$ is a nonnegative sequence in $\ell^2(\{\mu_I \}, \mathbb{C})$  or equivalently, $a$ (defined by $a(I) = a_I$) is a nonnegative function in $L^2(\mathcal{Z})$, and the norms of the two sequences are equal. It is easy to show that 
\[ \int_{\mathcal{Z}} \int_{\mathcal{Z}} k(K,L) a(K) a(L) \ dK dL = \sum_{J \in \mathcal{D}} \sum_{I: I \subseteq J} \mu_I \mu_J T_{JI} a_I a_J = \sum_{J \in \mathcal{D}} \sum_{I:I \subseteq J} \mu_I \mu_J T_{JI}  \| \alpha_I \|_{\mathbb{C}^d} \| \alpha_J \|_{\mathbb{C}^d}, \]
which is exactly the object we need to control. Indeed, if we can establish the conditions of the Senichkin-Vinogradov test with constant $C_1$, then the result will be proved. Let us first rewrite the desired conditions. The definition of $k$ implies that
\[ \int_{\mathcal{Z}} k(K,J) k(K,J') \ dK = \sum_{I: I \subseteq J, J'} T_{JI} T_{J'I} \mu_I \qquad \forall \ J, J' \in \mathcal{D}. \]
Again using the definition of $k$, we have
\[ k(J, J') + k(J', J) = T_{JJ'} + T_{J'J} \qquad \forall \ J, J' \in \mathcal{D}. \]
Since we only sum over dyadic $I \subseteq J \cap J'$, to have a nonzero sum, we must have $J \subseteq J'$ or $J' \subseteq J$. Without loss of generality, assume $J' \subseteq J.$ Then, to establish the conditions of the Senichkin-Vinogradov test, one must simple show:
\[ 
\begin{aligned}
\sum_{I: I \subseteq J'} T_{JI} T_{J'I} \mu_I &= \sum_{I: I \subseteq J'} \mu_I \frac{1}{|J|} \left \|   \left \langle W \right \rangle^{-\frac{1}{2}}_J  \left \langle W \right \rangle^{\frac{1}{2}}_I \right \| \frac{1}{|J'|} \left \|   \left \langle W \right \rangle^{-\frac{1}{2}}_{J'}  \left \langle W \right \rangle^{\frac{1}{2}}_I \right \| \\
& \le C_1 \frac{1}{|J|}
 \left \|   \left \langle W \right \rangle^{-\frac{1}{2}}_J \left \langle W \right \rangle^{\frac{1}{2}}_{J'} \right \|. 
\end{aligned}
  \]
This inequality is proven in detail in \cite{vt97}. The proof uses simple results about matrix weights including the fact that all matrix $A_2$ weights satisfy a reverse H\"older estimate as in \eqref{eqn:reverse}. The reverse H\"older estimate is used to turn the sum of interest into a sum of averages of a function weighted by the constants $\mu_I.$ Since $\{\mu_I\}_{I\in\mathcal{D}}$ is a scalar Carleson sequence, one can use the scalar Carleson Embedding Theorem to  complete the proof.
  \end{proof}

Using Theorem \ref{thm:CET1} and ideas from \cite{IKP}, we now obtain the following Carleson Embedding Theorem. Its testing conditions are particularly well-suited to the objects appearing in the proofs of Theorems \ref{thm:WellLoc} and \ref{thm:WellLoc2}, the well-localized analogues of Theorems \ref{thm:Band} and \ref{thm:Band2}. 

\begin{theorem} \label{thm:CET2} Let $W$ be an $A_2$ weight and let $\{A_I\}_{I\in\mathcal{D}}$ be a sequence of positive semi-definite $d \times d$ matrices. Then 
\[
\sum_{I\in\mathcal{D}} \left\langle A_I \left\langle f\right\rangle_I, \left\langle f\right\rangle_I\right\rangle_{\mathbb{C}^d} \le C_1 \left\Vert f\right\Vert_{L^2(W^{-1})}^2 \  \text{ if  } \ \
\frac{1}{|J|} \sum_{I: I \subseteq J} \left \langle W \right \rangle_I A_I  \left \langle W \right \rangle_I \le C_2  \left \langle W \right \rangle_J \ \ \ \forall J \in \mathcal{D},
\]
where $C_1 = C_2 C(d)[W]_{R_2} [W]_{A_2}.$
\end{theorem}

The existence of Theorem \ref{thm:CET2}, albeit with a different constant, is mentioned by Isralowitz-Kwon-Pott in the final remarks of \cite{IKP}. Indeed, according to these remarks, if one modifies their previous arguments and tracks all constants closely, one could obtain this Carleson Embedding Theorem with constant $C(d) [W]^2_{A_2}.$  However, in light of Equation $\eqref{eqn:RHconstant}$, our constant is very likely smaller than the one appearing in \cite{IKP}. As the details of the proof are not given in \cite{IKP} and we obtain a different constant, we include the proof here.

\begin{remark} \label{rem:CET} In Theorems \ref{thm:Band}, \ref{thm:Band2} and Theorems \ref{thm:WellLoc}, \ref{thm:WellLoc2}, the constants $B(W)$ and $B(V)$ appear. Since dimensional constants are already included in the statement of those theorems, it should be clear from Theorem \ref{thm:CET2} that
\[ B(W) = [W]_{R_2}^{\frac{1}{2}} [W]_{A_2}^{\frac{1}{2}}  \ \text{ and } \ B(V) = [V]_{R_2}^{\frac{1}{2}} [V]_{A_2}^{\frac{1}{2}}.\]
\end{remark} 

Now, to prove Theorem \ref{thm:CET2}, we need the decaying stopping tree from Isralowitz-Kwon-Pott. Specifically, fix $I \in \mathcal{D}$ and let $\mathcal{J}(I)$ be the collection of maximal dyadic $J \subseteq I$ such that
\[  \left \| \left \langle W \right \rangle_J^{-\frac{1}{2}} \left \langle W \right \rangle_I^{\frac{1}{2}} \right \|^2 > \lambda \ \ \text{ or } \ \ \left \| \left \langle W \right \rangle_J^{\frac{1}{2}} \left \langle W \right \rangle_I^{-\frac{1}{2}} \right \|^2 > \lambda, \]
for $\lambda >1$ to be determined later.  Set $\mathcal{F}(I)$ to be the collection of $J \subseteq I$ such that $J$ is not contained in any interval in $\mathcal{J}(I).$ It is clear that $I$ is always in $\mathcal{F}(I).$
Set $ \mathcal{J}^0(I) \equiv \{I \}.$ Inductively define $\mathcal{J}^j(I)$ and $\mathcal{F}^j(I)$ by 
\[ \mathcal{J}^j(I) = \bigcup_{J \in \mathcal{J}^{j-1}(I)} \mathcal{J}(J) \ \ \text { and } \ \ \mathcal{F}^j(I) = \bigcup_{J \in \mathcal{J}^{j-1}(I)} \mathcal{F}(J). \]

One can then prove the following lemma.
\begin{lemma}[Lemma 2.1, \cite{IKP}] \label{lem:stopping} Given the stopping-tree set-up, if $\lambda = 4 C(d)[W]_{A_2},$ then
\[ \left | \bigcup_{J \in \mathcal{J}^j(I)} \mathcal{J}(J) \right| \le 2^{-j} |I| \quad \forall I \in \mathcal{D}. \]
\end{lemma}

We can now provide the proof of Theorem \ref{thm:CET2}:

\begin{proof}[Proof of Theorem \ref{thm:CET2}] Using the equivalence, up to a dimensional constant, of  norm and trace for positive semi-definite matrices, our hypothesis implies
\[ \sum_{I: I \subseteq K} \left \| \left \langle W \right \rangle_K^{-\frac{1}{2}}  \left \langle W \right \rangle_I A_I  \left \langle W \right \rangle_I \left \langle W \right \rangle_K^{-\frac{1}{2}} \right \|  \lesssim C_2 |K| \quad \forall K \in \mathcal{D}. \]
We will use this to obtain the testing condition from Theorem \ref{thm:CET1}. Specifically, fix $J \in \mathcal{D}$. Then

\[
\begin{aligned}
\frac{1}{|J|} &\sum_{I: I \subseteq J} \left \| \left \langle W \right \rangle^{\frac{1}{2}}_I A_I  \left \langle W \right \rangle^{\frac{1}{2}}_I \right \| = 
\frac{1}{|J|} \sum_{j=1}^{\infty} \sum_{K \in \mathcal{J}^{j-1}(J)} \sum_{I \in \mathcal{F}(K)} \left \| \left \langle W \right \rangle^{\frac{1}{2}}_I A_I  \left \langle W \right \rangle^{\frac{1}{2}}_I \right \| \\
 &\le \frac{1}{|J|} \sum_{j=1}^{\infty} \sum_{K \in \mathcal{J}^{j-1}(J)} \sum_{I \in \mathcal{F}(K)} 
 \left \|  \left \langle W \right \rangle^{-\frac{1}{2}}_I \left \langle W \right \rangle^{\frac{1}{2}}_K \right \|
 \left \| \left \langle W \right \rangle^{-\frac{1}{2}}_K \left \langle W \right \rangle_I A_I  \left \langle W \right \rangle_I  \left \langle W \right \rangle^{-\frac{1}{2}}_K \right \|  \left \| \left \langle W \right \rangle^{\frac{1}{2}}_K \left \langle W \right \rangle^{-\frac{1}{2}}_I \right \| \\
 & =  \frac{1}{|J|} \sum_{j=1}^{\infty} \sum_{K \in \mathcal{J}^{j-1}(J)} \sum_{I \in \mathcal{F}(K)} 
 \left \| \left \langle W \right \rangle^{\frac{1}{2}}_K \left \langle W \right \rangle^{-\frac{1}{2}}_I \right \|^2
 \left \| \left \langle W \right \rangle^{-\frac{1}{2}}_K \left \langle W \right \rangle_I A_I  \left \langle W \right \rangle_I  \left \langle W \right \rangle^{-\frac{1}{2}}_K \right \| \\
 &\lesssim  \frac{  [W]_{A_2} }{|J|} \sum_{j=1}^{\infty} \sum_{K \in \mathcal{J}^{j-1}(J)} \sum_{I \in \mathcal{F}(K)} 
 \left \| \left \langle W \right \rangle^{-\frac{1}{2}}_K \left \langle W \right \rangle_I A_I  \left \langle W \right \rangle_I  \left \langle W \right \rangle^{-\frac{1}{2}}_K \right \| \\
 & \le\frac{ [W]_{A_2} }{|J|} \sum_{j=1}^{\infty} \sum_{K \in \mathcal{J}^{j-1}(J)} \sum_{I:I \subseteq K} 
 \left \| \left \langle W \right \rangle^{-\frac{1}{2}}_K \left \langle W \right \rangle_I A_I  \left \langle W \right \rangle_I  \left \langle W \right \rangle^{-\frac{1}{2}}_K \right \| \\
 & \lesssim  \frac{C_2 [W]_{A_2} }{|J|} \sum_{j=1}^{\infty} \sum_{K \in \mathcal{J}^{j-1}(J)} |K| \\
 & \le C_2  [W]_{A_2} \sum_{j=1}^{\infty}  2^{-j} \\
 & =  C_2 [W]_{A_2}.
\end{aligned}
\]
In the fourth line from the top we use the stopping criteria, which introduces the value $[W]_{A_2}$.  Pairing this estimate with Theorem \ref{thm:CET1} gives the desired result.
\end{proof}

\begin{remark} As mentioned in  \cite{IKP}, one can prove a version of Lemma \ref{lem:stopping} for $A_{2, \infty}$ weights using Lemma 3.1 in \cite{vol97}.   Recall from \cite{vol97} that $W$ is an  $A_{2, \infty}$ weight if there is some constant $C$ such that
\[ e^{ \displaystyle \tfrac{1}{|I|} \int_I \log \| W(t)^{-\frac{1}{2}} x\| dt} \le C  \left \| \left \langle W \right \rangle_I^{-\frac{1}{2}} x \right \|, \quad \forall x \in \mathbb{C}^d, \ I \in \mathcal{D}.\]
	Denote the smallest such $C$ by $[ W]_{A_{2,\infty}}$. As is shown in \cite{vol97}, if $W \in A_2$, then $W \in A_{2, \infty}$ with $[W]_{A_{2, \infty}} \le [W]_{A_2}.$ If one tracks the constant in Lemma 3.1 from \cite{vol97} and uses it in the proof of Lemma 2.1 in \cite{IKP}, one can obtain Lemma \ref{lem:stopping} with $\lambda =C(d) [W]_{A_{2,\infty}}^{2d}.$ 
Then the proof of Theorem \ref{thm:CET2} immediately shows that Theorem \ref{thm:CET2} also holds with constant $C_1 = C_2 C(d)[W]_{R_2} [W]_{A_{2, \infty}}^{2d}.$
\end{remark}

\section{Well-Localized Operators} \label{sec:WellLoc}
We say an operator $T_W$ acts formally from $L^2(W)$ to $L^2(V)$ if the bilinear form
\[ \left \langle T_W \textbf{1}_I e, \textbf{1}_J v \right \rangle_{{L^2(V)} } \]
is given for all $I,J\in\mathcal{D}$ and  $e, v \in \mathbb{C}^d$ is well-defined. Then, the formal adjoint $T_V^*$ is defined by
\[ \left \langle T_V^* \textbf{1}_I e, \textbf{1}_J v \right \rangle_{{L^2(W)} } \equiv  \left \langle \textbf{1}_I e, T_W \textbf{1}_J v \right \rangle_{{L^2(V)} }.\]
Given this, we can define:

\begin{definition} \label{def:wellloc} An operator $T_W$ acting (formally) from $L^2(W)$ to $L^2(V)$ is called \emph{$r$-lower triangular} if for all $ 1 \le j \le d$ and $I, J \in \mathcal{D}$ with $|J| \le 2 |I|$ and all $e\in\mathbb{C}^d$, $T_W$ satisfies
\[ \left \langle T_W \textbf{1}_I e, h^{V,j}_J \right \rangle_{L^2(V)}  =0\]
whenever $J \not \subset I^{(r+1)}$ or  $|J| \le 2^{-r}|I|$ and $J \not \subset I.$ Here,  $ \left\{ h^{V,j}_J \right\} $ is the set of  $V$-weighted Haar functions on $J$ as defined in \eqref{eqn:haarfunctions} and $I^{(r+1)}$ is the $(r+1)^{th}$ ancestor of $I$. We say $T_W$ is \emph{well-localized with radius $r$} if both $T_W$ and its formal adjoint $T_V^*$ are $r$-lower triangular.
\end{definition}
 
This definition of well-localized is slightly different than the one appearing in \cite{ntv08}. Indeed, to define lower triangular, Nazarov-Treil-Volberg only impose conditions on $T_W$ when $|J| \le  |I|,$ rather than $|J| \le 2 |I|.$ 
Nevertheless, their ideas are clearly the correct ones and their definition is essentially correct; the difference is likely attributable to a typographical error. Still, after the establishing the related proofs, we do point out the necessity of having conditions for $|J| \le 2 |I|$ in Remark \ref{remark:definition}. 

The main results about well-localized operators are the following two theorems, which are the well-localized analogues of Theorems \ref{thm:Band} and \ref{thm:Band2}:

\begin{theorem} \label{thm:WellLoc} Let $V,W$ be matrix $A_2$ weights, and assume $T_W$ is a well-localized operator of radius $r$ acting formally from $L^2(W)$ to $L^2(V).$ Then $T_W$ extends to a bounded operator from $L^2(W)$ to $L^2(V)$ if and only if
\[
\begin{aligned}
\left \| T_W \textbf{1}_I e \right \|_{L^2(V)} &\le A_1 \left \langle W(I) e, e \right \rangle_{\mathbb{C}^d}^{\frac{1}{2}} \\
 \left \| T^*_V \textbf{1}_I e \right \|_{L^2(W)} &\le A_2 \left \langle V(I) e, e \right \rangle_{\mathbb{C}^d}^{\frac{1}{2}}
 \end{aligned} 
 \]
 for all $I \in \mathcal{D}$ and $e \in \mathbb{C}^d$. Furthermore, 
 \[ \left \| T_W \right \|_{L^2(W) \rightarrow  L^2(V)} \le 2^{2r}C(d) \left(A_1B(W) + A_2B(V) \right), \]
 where $C(d)$ is a dimensional constant and $B(W)$ and $B(V)$ are constants depending on $W$ and $V$ from an application of the matrix Carleson Embedding Theorem.
 \end{theorem} 

\begin{theorem} \label{thm:WellLoc2}   Let $V,W$ be matrix $A_2$ weights, and assume $T_W$  is a well-localized operator of radius $r$ acting formally from $L^2(W)$ to $L^2(V).$ Then $T_W$ extends to a bounded operator from $L^2(W)$ to $L^2(V)$ 
 if and only if the following two conditions hold:
\begin{itemize} 
\item[$(i)$] For all intervals $I \in \mathcal{D}$ and $e \in \mathbb{C}^d$,
\begin{eqnarray*} 
\left \| \textbf{1}_I T_W \textbf{1}_I e \right \|_{L^2(V)} &\le A_1 \left \langle W(I) e, e \right \rangle_{\mathbb{C}^d}^{\frac{1}{2}} \\ 
 \left \| \textbf{1}_IT^*_V \textbf{1}_I e \right \|_{L^2(W)} &\le A_2 \left \langle V(I) e, e \right \rangle_{\mathbb{C}^d}^{\frac{1}{2}}.
\end{eqnarray*}
 \item[$(ii)$] For all intervals $I, J$ in $\mathcal{D}$ satisfying $2^{-r} |I| \le |J| \le 2^r |I|$ and vectors $e, \nu$ in $\mathbb{C}^d$,
 \[ \left| \left \langle T_W \textbf{1}_I e, \textbf{1}_J \nu \right \rangle_{L^2(V)} \right| \le
 A_3 \left \langle W(I)e,e \right \rangle^{\frac{1}{2}}_{\mathbb{C}^d}
 \left \langle V(J) \nu,\nu \right \rangle^{\frac{1}{2}}_{\mathbb{C}^d}.
 \]
 \end{itemize}
  Furthermore, 
 \[  \left \| T_W \right \|_{L^2(W) \rightarrow  L^2(V)} \le 2^{2r} C(d) \left( A_1B(W) + A_2 B(V) +A_3\right), \]
 where $C(d)$ is a dimensional constant and $B(W)$ and $B(V)$ are constants depending on $W$ and $V$ from an application of the matrix Carleson Embedding Theorem.
 \end{theorem}

Theorems \ref{thm:Band} and \ref{thm:Band2} will follow immediately from these theorems once we establish the following lemma:

\begin{lemma} \label{lem:wlb} If $V,W$ are matrix weights whose entries are in $L^2_{loc}(\mathbb{R})$ and if $T$ is a band operator of radius $r$, then $T_W$ is a well-localized operator of radius $r$ acting formally from $L^2(W)$ to $L^2(V).$ \end{lemma}

\begin{proof} Assume $T: L^2 \rightarrow L^2$ is a band operator with radius $r$, and $W,V$ are matrix weights whose entries are in $L^2_{loc}$. Then the operators
\[ T_W \equiv T M_W \ \text{ and } \ T_V^* \equiv T^* M_V \]
act formally from $L^2(W)$ to $L^2(V)$ and $L^2(V)$ to $L^2(W)$ respectively since
\[ \left \langle T W \textbf{1}_I e, V \textbf{1}_J \nu \right \rangle_{L^2} = \left \langle T_W \textbf{1}_I e,  \textbf{1}_J \nu \right \rangle_{L^2(V)}
\text{ and } \left \langle W \textbf{1}_I e, T^*V \textbf{1}_J \nu \right \rangle_{L^2} = \left \langle  \textbf{1}_I e, T^*_V \textbf{1}_J \nu \right \rangle_{L^2(W)}
\]
are well-defined. To show $T_W$ is a well-localized operator with radius $r,$ by symmetry, it suffices to show that $T_W$ is $r$-lower triangular. First, fix an orthonormal basis $\{e_i \}_{i=1}^d$ of $\mathbb{C}^d$ and for $I \in \mathcal{D}$, define $H_I \equiv \{ h_I e_i \}_{1 \le i \le d}$. Then we can write
\[ T = \sum_{I,J \in \mathcal{D}} T_{IJ} \ \text{ where } \ T_{IJ} : H_I \rightarrow H_J, \]
and each $T_{IJ}$ is given by 
\[ T_{IJ} = \sum_{1 \le i,j \le d} \left \langle T h_I e_i, h_J e_j \right \rangle_{L^2} \left \langle \cdot, h_I e_i \right \rangle_{L^2} h_Je_j .\]
Since the entries of $W$ are in $L^2_{loc}(\mathbb{R})$, then  $W\textbf{1}_Ie$ is in $L^2$ and so, $T_W \textbf{1}_I e \equiv T W\textbf{1}_I e$ makes sense for each $I \in \mathcal{D}$ and $ e \in \mathbb{C}^d.$ Given $h^{V,j}_J$, a vector-valued Haar function on $J$ adapted to $V$, one can write: 
\[ \left \langle T_W \textbf{1}_I e, h_J^{V,j} \right \rangle_{L^2(V)} = \left \langle T_W \textbf{1}_I e,  V h_J^{V,j} \right \rangle_{L^2} \le \left \|  T_W \textbf{1}_I e \right \|_{L^2} \left \| V h_J^{V,j} \right \|_{L^2} < \infty,
 \] 
where the first term is bounded because $T$ is bounded on $L^2$ and the second term is bounded because $h_J^{V,j}$ is bounded and the entries of $V$ are in $L^2_{loc}(\mathbb{R}).$ Given that, we are justified in expanding $T$ with respect to the Haar basis to obtain
\[ 
\begin{aligned}
\left \langle T_W \textbf{1}_I e, h_J^{V,j} \right \rangle_{L^2(V)} &= \sum_{K,L \in \mathcal{D}} \left \langle T_{KL} W \textbf{1}_I e, h^{V,j}_J \right \rangle_{L^2(V)}  \\  
&= \sum_{K,L \in \mathcal{D}}\sum_{1 \le k,\ell \le d} \left \langle T h_K e_k, h_L e_{\ell} \right \rangle_{L^2} \left \langle W \textbf{1}_I e, h_K e_k \right \rangle_{L^2}  \left \langle h_Le_{\ell},  h^{V,j}_J \right \rangle_{L^2(V)}  .
\end{aligned} 
\]
Observe that $ \left \langle T_{KL} W \textbf{1}_I e, h^{V,j}_J \right \rangle_{L^2(V)} $ is zero if $d_{\text{tree}}(K,L) >r$, if $I \cap K = \emptyset,$ or if $L \not \subset J.$ So, we only need consider terms where $d_{\text{tree}}(K,L)  \le r$, $I \cap K \ne \emptyset$, and $L  \subseteq J.$ 

To show $T_W$ is $r$-lower triangular let $|J| \le 2 |I|$. 
First, assume that $J \not \subset I^{(r+1)}$ and by contradiction, assume there is a nonzero term $\left \langle T_{KL} W \textbf{1}_I e, h^{V,j}_J \right \rangle_{L^2(V)} $ in the above sum for some $K,L \in \mathcal{D}.$ By our previous assertions, we must have  
\[ |K| \le 2^r |L| \le 2^r |J | \le 2^{r+1} |I|.\]
Since $I \cap K \ne \emptyset$, this implies that $K \subseteq I^{(r+1)}.$ Since $L \subseteq J$, $|L| \le 2 |I|$ and $L \not \subset I^{(r+1)}.$  But, this immediately implies that $d_{tree}(K,L) \ge r+1$, a contradiction.

Similarly, assume $|J| \le 2^{-r} |I|$ and $J \not  \subset I$ and by contradiction, assume there is a nonzero term $\left \langle T_{KL} W \textbf{1}_I e, h^{V,j}_J \right \rangle_{L^2(V)} $ for some $K,L.$
Then $|L| \le 2^{-r} |I|$ and $L \not \subset I.$ Furthermore, since $d_{\text{tree}}(K,L) \le r$, this implies $|K| \le |I|$, so $K \subseteq I.$ But  $|L| \le 2^{-r} |I|$, $L \not \subset I$, and $K \subseteq I$ implies that $d_{\text{tree}}(K,L) \ge r+1$, a contradiction.

Thus, $T_W$ is $r$-lower triangular and symmetric arguments give the result for $T_V^*$. This implies $T_W$ is well-localized with radius $r$.
\end{proof}

\begin{remark} In Theorems \ref{thm:WellLoc} and \ref{thm:WellLoc2}, one must interpret the testing conditions correctly when the matrix weights' entries are not in $L^2_{loc}(\mathbb{R})$. We already outlined the remedy for this problem in Remark \ref{rem:L1loc}. Similarly, one should notice that Lemma \ref{lem:wlb} only handles the case where the matrix weights have entries in $L^2_{loc}(\mathbb{R})$. Nevertheless, this result is sufficient to allow us to pass from Theorems \ref{thm:WellLoc} and \ref{thm:WellLoc2} to Theorems \ref{thm:Band} and \ref{thm:Band2}. This is easy to see since, as detailed in Remark \ref{rem:L1loc}, we interpret all statements about weights with locally integrable (but not necessary square-integrable) entries in Theorems \ref{thm:Band} and \ref{thm:Band2} using limits of weights with entries in $L^2_{loc}(\mathbb{R})$.
\end{remark}

\section{Proofs of Theorems \ref{thm:WellLoc} and \ref{thm:WellLoc2}} \label{sec:proof}

\subsection{Paraproducts} \label{sec:paraproducts}

To prove Theorems \ref{thm:WellLoc} and \ref{thm:WellLoc2}, we require several results about related paraproducts. As before, let $T_W$ be a well-localized operator of radius $r$ acting formally from $L^2(W)$ to $L^2(V)$ with formal adjoint $T^*_V.$ 
Using these operators, define the following paraproducts:
\[ 
\begin{aligned}
\Pi^W f &\equiv \sum_{I \in \mathcal{D}} \sum_{\substack{ 1 \le j \le d \\ J \subseteq I : |J| = 2^{-r} |I|}}
\left \langle T_W E^W_I f , h^{V,j}_J \right \rangle_{L^2(V)}  h^{V,j}_J \\
\Pi^V g &\equiv \sum_{I \in \mathcal{D}} \sum_{\substack{ 1 \le j \le d \\ J \subseteq I : |J| = 2^{-r} |I|}}
\left \langle T^*_V  E^V_I g, h^{W,j}_J \right \rangle_{L^2(W)}  h^{W,j}_J 
\end{aligned}
\]
for $f \in L^2(W)$ and $ g \in L^2(V)$. Recall that the $W$-weighted expectation of $f$ on $I$ is defined by $E^W_If \equiv \left \langle W \right \rangle_I^{-1} \left \langle Wf \right \rangle_I \textbf{1}_I.$ Now, observe that, as demonstrated by the following lemma, these paraproducts mimic the behavior of $T_W$ and $T^*_V$ respectively.

\begin{lemma} \label{lem:paraproduct} Let $I, J \in \mathcal{D}$ and let $\Pi^W$ be the paraproduct defined above using the well-localized operator $T_W$ with radius $r$ acting (formally) from $L^2(W)$ to $L^2(V)$. 
\begin{itemize}
\item[1.] If $|J| \ge 2^{-r} |I|$, then 
\[ \left \langle \Pi^W h^{W,i}_I, h^{V,j}_J \right \rangle_{L^2(V)} = 0 \qquad \forall \ 1 \le i,j \le d.\]

\item[2.] If $|J| < 2^{-r} |I|$, then 
\[ \left \langle \Pi^W h^{W,i}_I, h^{V,j}_J \right \rangle_{L^2(V)} = \left \langle T_W h^{W,i}_I, h^{V,j}_J \right \rangle_{L^2(V)}  \qquad \forall \ 1 \le i,j \le d.\]
If $J  \not \subset I$, then both sides of the equality are zero.
\end{itemize}
Furthermore, analogous statements hold for the paraproduct $\Pi^V$ and formal adjoint $T^*_V.$
\end{lemma}

\begin{proof} First, observe that 
\[ 
\begin{aligned}
\left \langle \Pi^W h_I^{W,i}, h_J^{V,j} \right \rangle_{L^2(V)} &= 
\sum_{K \in \mathcal{D}} \sum_{\substack{ 1 \le \ell \le d \\ L \subseteq K: |L| = 2^{-r} |K|} } 
\left \langle T_W E^W_K h^{W,i}_I, h^{V,\ell}_L \right \rangle_{L^2(V)} \left \langle h^{V, \ell}_L, h^{V,j}_J \right \rangle_{L^2(V)} \\
&= \left \langle T_W E^W_{J^{(r)}} h^{W,i}_I, h^{V,j}_J \right \rangle_{L^2(V)},
\end{aligned}
\]
where $J^{(r)}$ is the $r^{th}$ ancestor of $J$. Now assume $|J| \ge 2^{-r} |I|$ or $J \not \subset I.$ Then, either $I \subseteq J^{(r)}$ or $I \cap J^{(r)} = \emptyset.$ In either case, 
\[ E^W_{J^{(r)}} h_I^{W,i} =0,\]
so  the corresponding inner product is zero. Now assume $|J| < 2^{-r} |I|$, so that $|J| \le 2^{-r} |I_-| =2^{-r} |I_+|$. If $J \not \subset I$, then $J \not \subset I_-, I_+$ and since $T_W$ is well-localized with radius $r$,
\[ \left \langle T_W h_I^{W,i}, h_J^{V,j} \right \rangle_{L^2(V)} =  \left \langle T_W h_I^{W,i}(I_-)  \textbf{1}_{I_-}, h_J^{V,j} \right \rangle_{L^2(V)} + \left \langle T_W h_I^{W,i}(I_+) \textbf{1}_{I_+}, h_J^{V,j} \right \rangle_{L^2(V)} =0.\]
This gives equality if $J \not \subset I.$ Now assume $|J| < 2^{-r} |I|$ and $J \subseteq I.$ Then
\[
\begin{aligned}
\left \langle \Pi^W h_I^{W,i}, h_J^{V,j} \right \rangle_{L^2(V)} &= \left \langle T_W E^W_{J^{(r)}} h^{W,i}_I, h^{V,j}_J \right \rangle_{L^2(V)} \\
&= \left \langle T_W h^{W,i}_I \left(J^{(r)} \right)  \textbf{1}_{J^{(r)}}, h^{V,j}_J \right \rangle_{L^2(V)} \\
&= \left \langle T_W h^{W,i}_I, h^{V,j}_J \right \rangle_{L^2(V)},
\end{aligned}
\]
since for all $I' \subset I \setminus J^{(r)}$, the tree distance $d_{\text{tree}}(I',J) >r$ and so
\[ \left \langle T_W  h^{W,i}_I \left(I' \right) \textbf{1}_{I'}, h_J^{V,j} \right \rangle_{L^2(V)} = 0.\]
Analogous statements hold for $\Pi^V$, since it is defined using the operator $T^*_V$, which is also well-localized with radius $r$. 
\end{proof}

Now, we show that the testing condition $(i)$ from Theorem \ref{thm:WellLoc2} and hence, the stronger testing condition from Theorem \ref{thm:WellLoc}, implies the boundedness of the paraproducts $\Pi^W$ and $\Pi^V.$ We state the result for $\Pi^W,$ but analogous arguments give the result for $\Pi^V.$

\begin{lemma}  \label{lem:parbdd} Let $\Pi^W$ be the paraproduct defined above and assume that the well-localized operator $T_W$ satisfies:
\[
\left \|  \textbf{1}_I T_W \textbf{1}_I e \right \|_{L^2(V)} \le C \left \langle W(I) e,e \right \rangle_{\mathbb{C}^d}^{\frac{1}{2}}  \qquad \forall I \in \mathcal{D}, \ e \in \mathbb{C}^d.\]
Then $\Pi^W$ is bounded from $L^2(W)$ to $L^2(V)$ and 
\[  \left \| \Pi^W \right \|_{L^2(W) \rightarrow L^2(V)} \le C B(W), \] 
where $B(W)$ is the constant obtained from applying the matrix Carleson Embedding Theorem.
\end{lemma}

\begin{proof} Fix $f \in L^2(W)$, which implies $Wf \in L^2(W^{-1})$, and observe that 
\[ 
\begin{aligned}
\left \| \Pi^W f \right \|_{L^2(V)}^2 &= \sum_{K \in \mathcal{D}}  \sum_{\substack{ 1 \le \ell \le d \\ L \subseteq K: |L| = 2^{-r} |K|} }  \left | \left \langle T_W E^W_K f, h^{V,\ell}_L \right \rangle_{L^2(V)} \right|^2 \\
& = \sum_{K \in \mathcal{D}}  \sum_{\substack{ 1 \le \ell \le d \\ L \subseteq K: |L| = 2^{-r} |K|} }
 \left | \left  \langle  E^W_K f, T_V^*h^{V,\ell}_L \right \rangle_{L^2(W)} \right|^2 \\
 & =\sum_{K \in \mathcal{D}}  \sum_{\substack{ 1 \le \ell \le d \\ L \subseteq K: |L| = 2^{-r} |K|} }
 \left | \left  \langle  \left \langle W \right \rangle_K^{-1} \left \langle Wf  \right \rangle_K , \alpha_{L, \ell} \right \rangle_{\mathbb{C}^d} \right|^2,  
\end{aligned}
\]
where we have set $\alpha_{L, \ell}$ to be the vector
\[ \alpha_{L, \ell} \equiv \int_{L^{(r)}} W(x) T^*_V h^{V,\ell}_L(x) dx. \]
And so, letting $(\alpha_{L, \ell})^*$ denote the $1 \times d$ adjoint row vector corresponding to $\alpha_{L, \ell}$, we have
\[ 
\begin{aligned}
\left \| \Pi^W f \right \|_{L^2(V)}^2 &= 
\sum_{K \in \mathcal{D}}  \sum_{\substack{ 1 \le \ell \le d \\ L \subseteq K: |L| = 2^{-r} |K|} }
\left \langle  \alpha_{L, \ell} \left( \alpha_{L, \ell} \right)^*  \left \langle W \right \rangle_K^{-1} \left \langle Wf  \right \rangle_K , \left \langle W \right \rangle_K^{-1} \left \langle Wf  \right \rangle_K \right \rangle_{\mathbb{C}^d} \\
&= \sum_{K \in \mathcal{D}} \left \langle A_K \left \langle Wf  \right \rangle_K ,\left \langle Wf  \right \rangle_K \right \rangle_{\mathbb{C}^d},
\end{aligned}
\]
where we have set
\[ A_K \equiv  \sum_{\substack{ 1 \le \ell \le d \\ L \subseteq K: |L| = 2^{-r} |K|} }  \left \langle W \right \rangle_K^{-1}\alpha_{L, \ell} \left( \alpha_{L, \ell} \right)^*   \left \langle W \right \rangle_K^{-1}.\]
This is exactly the setup where we can apply Theorem \ref{thm:CET2}. Specifically, we need to show that for all $J \in \mathcal{D}$, 
\[ \sum_{K \subseteq J}  \left \langle W \right \rangle_K A_K  \left \langle W \right \rangle_K \le C^2 W(J).\]
To prove this matrix inequality, fix $e \in \mathbb{C}^d$ and observe that
\[
\begin{aligned}
\sum_{K \subseteq J}  \left \langle \left \langle W \right \rangle_K A_K  \left \langle W \right \rangle_K e,e
\right \rangle_{\mathbb{C}^d} &= \sum_{K \subseteq J} \sum_{\substack{ 1 \le \ell \le d \\ L \subseteq K: |L| = 2^{-r} |K|} }  \left \langle\alpha_{L, \ell} \left( \alpha_{L, \ell} \right)^* e,e
\right \rangle_{\mathbb{C}^d} \\
& = \sum_{K \subseteq J} \sum_{\substack{ 1 \le \ell \le d \\ L \subseteq K: |L| = 2^{-r} |K|} } \left| \left \langle\alpha_{L, \ell} ,e
\right \rangle_{\mathbb{C}^d} \right |^2 \\
& = \sum_{K \subseteq J} \sum_{\substack{ 1 \le \ell \le d \\ L \subseteq K: |L| = 2^{-r} |K|} }
\left | \left \langle h^{V, \ell}_L, T_W e \textbf{1}_K \right \rangle_{L^2(V)} \right|^2. 
\end{aligned}
\]
Notice that as $T$ is $r$-lower triangular and $L \subseteq K$ with $|L| = 2^{-r}|K|$, we have that
\[ \left \langle h_L^{V,\ell}, T_W e \textbf{1}_{J \setminus K} \right \rangle_{L^2(V)} = \sum_{ I \subseteq J: I \ne K, |I| = |K|} \left \langle h_L^{V,\ell}, T_W e \textbf{1}_{I} \right \rangle_{L^2(V)} 
=0.\]
This means that
\[
\begin{aligned}
\sum_{K \subseteq J}  \left \langle \left \langle W \right \rangle_K A_K  \left \langle W \right \rangle_K e,e
\right \rangle_{\mathbb{C}^d}
&= \sum_{K \subseteq J} \sum_{\substack{ 1 \le \ell \le d \\ L \subseteq K: |L| = 2^{-r} |K|} }
 \left| \left \langle h^{V, \ell}_L, T_W e \textbf{1}_J \right \rangle_{L^2(V)} \right|^2 \\
 & \le \left \| \textbf{1}_J T_W e \textbf{1}_J \right \|^2_{L^2(V)} \\
&\le C^2 \left \langle W(J)e, e \right \rangle_{\mathbb{C}^d}.
\end{aligned}
\] 
Since $e \in \mathbb{C}^d$ was arbitrary, the matrix inequality follows, so we can apply Theorem \ref{thm:CET2} to obtain:
\[ \| \Pi^W f \|^2_{L^2(V)} = \sum_{K \in \mathcal{D}} \left \langle A_K \left \langle W f \right \rangle_K, 
\left \langle W f \right \rangle_K \right \rangle_{\mathbb{C}^d} \le C^2 B(W)^2 \| W f \|^2_{L^2(W^{-1})} = C^2 B(W)^2 \|f \|^2_{L^2(W)},\]
as desired.
\end{proof}

\subsection{Small Lemmas}

In this subsection, we verify several small lemmas that are trivial in the scalar situation.  As before, $T_W$ is a well-localized operator with radius $r$ that satisfies the testing conditions from Theorem \ref{thm:WellLoc} or \ref{thm:WellLoc2}.

\begin{lemma} \label{lem:weightedbdd} Let $T_W$ be a well-localized operator with radius $r$ acting (formally) from $L^2(W)$ to $L^2(V)$ that satisfies the testing condition from Theorem \ref{thm:WellLoc} with constant $A_1$.  Then
\[  \left | \left \langle T_W h^{W,i}_I, h^{V,j}_J \right \rangle_{L^2(V)} \right | \le C(d) A_1 \qquad \forall I,J \in \mathcal{D},  1 \le i,j \le d.\]
Similarly, if $T_W$ satisfies the testing condition $(ii)$ from Theorem \ref{thm:WellLoc2} with constant $A_3$, then 
\[  \left | \left \langle T_W h^{W,i}_I, h^{V,j}_J \right \rangle_{L^2(V)} \right | \le C(d) A_3 \qquad \forall I,J \in \mathcal{D},  1 \le i,j \le d.\]
\end{lemma}
\begin{proof} For the first part of the lemma, we can use Cauchy-Schwarz to obtain:
\[ 
\begin{aligned}
\left | \left \langle T_W h^{W,i}_I, h^{V,j}_J \right \rangle_{L^2(V)} \right |  \le \left \| T_W h^{W,i}_I \right \|_{L^2(V)} \le  \left\| T_W h^{W,i}_I(I_-) \textbf{1}_{I_-}  \right \|_{L^2(V)} +  \left \| T_W h^{W,i}_I(I_+) \textbf{1}_{I_+} \right \|_{L^2(V)}.
\end{aligned} \]
It suffices to prove the desired bound for one term in the sum, since the arguments are symmetric. Using the testing condition and Lemma \ref{lem:haarbound}, we have:
\[ \begin{aligned}
\left \| T_W h_I^{W,i} \left( I_- \right ) \textbf{1}_{I_-} \right \|_{L^2(V)} &\le A_1 \left \langle W(I_-)  h_I^{W,i} \left( I_- \right ), h_I^{W,i} \left( I_- \right ) \right \rangle_{\mathbb{C}^d}^{\frac{1}{2}} \\
&= A_1\left \| W(I_-)^{\frac{1}{2}} h_I^{W,i} \left( I_- \right ) \right \|_{\mathbb{C}^d} \\
& \le C(d) A_1, 
\end{aligned}
\]
which completes the first part of the lemma. For the second part, we can write:
\[ 
\begin{aligned}
\left | \left \langle T_W h^{W,i}_I, h^{V,j}_J \right \rangle_{L^2(V)} \right | &\le 
 \left | \left \langle T_W h^{W,i}_I (I_-) \textbf{1}_{I_-}, h^{V,j}_J(J_-) \textbf{1}_{J_-} \right \rangle_{L^2(V)} \right |  +   \left | \left \langle T_W h^{W,i}_I (I_-) \textbf{1}_{I_-}, h^{V,j}_J(J_+) \textbf{1}_{J_+} \right \rangle_{L^2(V)} \right | \\
 & \ \ \  + \left | \left \langle T_W h^{W,i}_I (I_+) \textbf{1}_{I_+}, h^{V,j}_J(J_-) \textbf{1}_{J_-} \right \rangle_{L^2(V)} \right | + \left | \left \langle T_W h^{W,i}_I (I_+) \textbf{1}_{I_+}, h^{V,j}_J(J_+) \textbf{1}_{J_+} \right \rangle_{L^2(V)} \right |.  
\end{aligned}
 \]
By Lemma \ref{lem:haarbound} and testing hypothesis $(ii)$, we can conclude:
\[ 
\begin{aligned}
\left | \left \langle T_W h^{W,i}_I (I_-) \textbf{1}_{I_-}, h^{V,j}_J(J_-) \textbf{1}_{J_-} \right \rangle_{L^2(V)} \right | & \le A_3 \left \langle W(I_-) h^{W,i}_I (I_-) , h^{W,i}_I (I_-)  \right \rangle_{\mathbb{C}^d}^{\frac{1}{2}} \left \langle V(I_-) h^{V,j}_J(J_-), h^{V,j}_J(J_-)  \right \rangle_{\mathbb{C}^d}^{\frac{1}{2}} \\
& = A_3\left \| W(I_-)^{\frac{1}{2}} h_I^{W,i} \left( I_- \right ) \right \|_{\mathbb{C}^d} \left \| V(I_-)^{\frac{1}{2}} h_J^{V,j} \left( I_- \right ) \right \|_{\mathbb{C}^d} \\
&\le C(d) A_3.
\end{aligned}\]
The other three terms in the sum can be handled similarly.
\end{proof}

\begin{lemma} \label{lem:expectbd} Let $f \in  L^2(W)$. Then for all $I \in \mathcal{D}$, 
\[ |I|^{\frac{1}{2}} \left \| \left \langle W \right \rangle_I^{-\frac{1}{2}} \left \langle Wf \right \rangle_I \right
\|_{\mathbb{C}^d} \le C(d) \left \| f \textbf{1}_I  \right \|_{L^2(W)}. \]
\end{lemma}
\begin{proof} Using H\"older's inequality and the fact that $\left \langle W \right \rangle_I^{-\frac{1}{2}} W(x)\left \langle W \right \rangle_I^{-\frac{1}{2}}$ is positive a.e., we can compute 
\[
\begin{aligned}
|I| \left \| \left \langle W \right \rangle_I^{-\frac{1}{2}} \left \langle Wf \right \rangle_I \right
\|^2_{\mathbb{C}^d} & =|I|^{-1}\left \| \int_I \left \langle W \right \rangle_I^{-\frac{1}{2}} W(x) f(x) \ dx \right \|_{\mathbb{C}^d}^2 \\
&\le |I|^{-1} \left(\int_I \left \| \left \langle W \right \rangle_I^{-\frac{1}{2}} W(x) f(x) \right \|_{\mathbb{C}^d}dx \right)^2 \\
&\le |I|^{-1}\left(\int_I \left \| \left \langle W \right \rangle_I^{-\frac{1}{2}} W(x)^{\frac{1}{2}} \right \|^2 dx \right) \left( \int_I  \left \| W(x)^{\frac{1}{2}} f(x) \right \|^2_{\mathbb{C}^d}dx \right) \\
& = \left(|I|^{-1} \int_I   \left \| \left \langle W \right \rangle_I^{-\frac{1}{2}} W(x) \left \langle W \right \rangle_I^{-\frac{1}{2}}\right \|  dx \right) \left \| f \textbf{1}_I  \right \|^2_{L^2(W)}  \\
& \le C(d) \left \| f \textbf{1}_I  \right \|^2_{L^2(W)} \left \| |I|^{-1} \int_I \left \langle W \right \rangle_I^{-\frac{1}{2}} W(x)\left \langle W \right \rangle_I^{-\frac{1}{2}}  dx  \right \| \\
&= C(d) \left \| f \textbf{1}_I  \right \|^2_{L^2(W)} ,
\end{aligned}
\]
which gives the needed inequality. \end{proof}

\subsection{Proofs of Theorems \ref{thm:WellLoc} and \ref{thm:WellLoc2}} 

We first prove Theorem \ref{thm:WellLoc}:
\begin{proof}
We prove $T_W$ extends to a bounded operator from $L^2(W)$ to $L^2(V)$ using duality. Specifically we show
\begin{equation} \label{eqn:opineq} \left\vert \left \langle T_W f, g \right \rangle_{L^2(V)}\right\vert \le C\| f \|_{L^2(W)} \| g \|_{L^2(V)}, \end{equation}
for a fixed constant $C$ and all $f$ and $g$ in dense sets of $L^2(W)$ and $L^2(V)$ respectively.  Without loss of generality, we can assume $f$ and $g$ are compactly supported and so, we can choose disjoint $I_1, I_2 \in \mathcal{D}$
such that $\text{supp}(f), \text{supp}(g) \subseteq I_1 \cup I_2$ and $|I_1 | = |I_2 | =2^m$, for some $m \in \mathbb{N}.$ Using \eqref{eqn:sum}, we can write
\begin{eqnarray} \label{eqn:fdecomp}
f &=& f_1 + f_2 = \sum_{\substack{I: |I| \le 2^m \\ 1 \le i \le d}} \left \langle f, h^{W,i}_I \right \rangle_{L^2(W)} h^{W,i}_I + \sum_{k=1}^2 E^W_{I_k} f \\
g &=& g_1 + g_2 = \sum_{\substack{J: |J| \le 2^m \\ 1 \le j \le d}} \left \langle g, h^{V,j}_J \right \rangle_{L^2(V)} h^{V,j}_J + \sum_{\ell=1}^2 E^V_{I_\ell} g.  \label{eqn:gdecomp}
\end{eqnarray}
Using these decompositions, it suffices to show
\[ \left\vert \left \langle T_W f_i, g_j \right \rangle_{L^2(V)}\right\vert \le C\| f \|_{L^2(W)} \| g \|_{L^2(V)} \qquad \forall 1 \le i,j\le 2. \]
First, consider $f_1$ and $g_1$. Using Lemma \ref{lem:paraproduct}, we can write
\[
\begin{aligned}
\left \langle T_W f_1, g_1 \right \rangle_{L^2(V)} &=    \sum_{\substack{I: |I| \le 2^m \\ 1 \le i \le d}} \sum_{\substack{J: |J| \le 2^m \\ 1 \le j \le d}} \left \langle f, h^{W,i}_I \right \rangle_{L^2(W)} \left \langle g, h^{V,j}_J \right \rangle_{L^2(V)} \left \langle T_W h^{W,i}_I, h^{V,j}_J \right \rangle_{L^2(V)} \\
& = \sum_{\substack{I: |I| \le 2^m \\ 1 \le i \le d}} \sum_{\substack{J: |J| \le 2^m \\ |J| <2^{-r} |I| \\ 1 \le j \le d}}\left \langle f, h^{W,i}_I \right \rangle_{L^2(W)} \left \langle g, h^{V,j}_J \right \rangle_{L^2(V)} \left \langle T_W h^{W,i}_I, h^{V,j}_J \right \rangle_{L^2(V)} \\
&\ \ \ \  +  \sum_{\substack{J: |J| \le 2^m \\ 1 \le j \le d}} \sum_{\substack{ I: |I| \le 2^m \\ |I| <2^{-r} |J| \\ 1 \le i \le d}}
\left \langle f, h^{W,i}_I \right \rangle_{L^2(W)} \left \langle g, h^{V,j}_J \right \rangle_{L^2(V)} \left \langle T_W h^{W,i}_I, h^{V,j}_J \right \rangle_{L^2(V)} \\
&\ \ \ \   + \sum_{\substack{I: |I| \le 2^m \\ 1 \le i \le d}} \sum_{\substack{J: |J| \le 2^m\\   2^{-r}|I| \le  |J| \le 2^{r} |I| \\ 1 \le j \le d}}\left \langle f, h^{W,i}_I \right \rangle_{L^2(W)} \left \langle g, h^{V,j}_J \right \rangle_{L^2(V)} \left \langle T_W h^{W,i}_I, h^{V,j}_J \right \rangle_{L^2(V)} \\
& = \left \langle \Pi^W f_1, g_1 \right \rangle_{L^2(V)} + \left \langle  f_1, \Pi^V g_1 \right \rangle_{L^2(W)}  \\
& \ \ \ \ +  \sum_{\substack{I: |I| \le 2^m \\ 1 \le i \le d}} \sum_{\substack{J: |J| \le 2^m \\ 2^{-r}|I| \le  |J| \le 2^{r} |I| \\ 1 \le j \le d}}\left \langle f, h^{W,i}_I \right \rangle_{L^2(W)} \left \langle g, h^{V,j}_J \right \rangle_{L^2(V)} \left \langle T_W h^{W,i}_I, h^{V,j}_J \right \rangle_{L^2(V)}. 
\end{aligned}
\]
Lemma \ref{lem:parbdd} implies that
\[ \left |  \left \langle \Pi^W f_1, g_1 \right \rangle_{L^2(V)} \right | + \left | \left \langle  f_1, \Pi^V g_1 \right 
\rangle_{L^2(W)} \right| \le \left(A_1B(W) + A_2 B(V) \right) \| f \|_{L^2(W)} \| g \|_{L^2(V)}. \]
So, we just need to bound the last sum. We first apply Cauchy-Schwarz and exploit symmetry in the sums to obtain:
\begin{align} 
& \sum_{\substack{I: |I| \le 2^m \\ 1 \le i \le d}} \sum_{\substack{J: |J| \le 2^m \\ 2^{-r}|I| \le  |J| \le 2^{r} |I| \\ 1 \le j \le d}}  \left | \left \langle f, h^{W,i}_I \right \rangle_{L^2(W)} \left \langle g, h^{V,j}_J \right \rangle_{L^2(V)} \left \langle T_W h^{W,i}_I, h^{V,j}_J \right \rangle_{L^2(V)}  \right |  \nonumber \\ 
& \le 
\label{eqn:diagonal2} \left( \sum_{\substack{I: |I| \le 2^m \\ 1 \le i \le d}} \sum_{\substack{J: |J| \le 2^m \\ 2^{-r}|I| \le  |J| \le 2^{r} |I| \\ 1 \le j \le d}} \left | \left \langle f, h^{W,i}_I \right \rangle_{L^2(W)}\right| ^2  \left |\left \langle T_W h^{W,i}_I, h^{V,j}_J \right \rangle_{L^2(V)}  \right |  \right)^{1/2}  \\
 & \ \ \times \left(
\sum_{\substack{J: |J| \le 2^m \\ 1 \le j \le d}} \sum_{\substack{I:|I| \le 2^m \\ 2^{-r}|J| \le  |I| \le 2^{r} |J| \\ 1 \le i \le d}} \left |  \left \langle g, h^{V,j}_J \right \rangle_{L^2(V)} \right|^2 \left | \left \langle T_W h^{W,i}_I, h^{V,j}_J \right \rangle_{L^2(V)}  \right | \right)^{1/2} \nonumber.
\end{align}
Now, fix $I \in \mathcal{D}$. Since $T_W$ is well-localized, it is not hard to show that  there are only finitely many $J$ satisfying $2^{-r}|I| \le  |J| \le 2^{r} |I|$ such that 
\[  \left \langle T_W h^{W,i}_I, h^{V,j}_J \right \rangle_{L^2(V)} \ne 0.\]
Specifically, the number of such $J$ will always be bounded by a fixed constant times $2^{2r}$. Similarly, if we fix $J$, there are only finitely many $I$ satisfying $2^{-r}|J| \le  |I| \le 2^{r} |J|$ such that 
\[  \left \langle T_W h^{W,i}_I, h^{V,j}_J \right \rangle_{L^2(W)} =\left \langle h^{W,i}_I, T_V^*h^{V,j}_J \right \rangle_{L^2(V)} \ne 0.\]
The number of such $I$ will also be bounded by a fixed constant times $2^{2r}$. 
Thus, we can use the testing conditions and Lemma \ref{lem:weightedbdd} to estimate
\[
\eqref{eqn:diagonal2} \le A_1 2^{2r} C(d) \| f \|_{L^2(W)} \|g\|_{L^2(V)}.
\]
The other terms are much simpler. First observe that for each $k, \ell$:
\[ 
\begin{aligned}
\left | \left \langle T_W E^W_{I_k} f , E^V_{I_{\ell}} g  \right \rangle_{L^2(V)}  \right | &\le  \left \|  T_W E^W_{I_k} f \right \|_{L^2(V)}    \left \| \left \langle V \right \rangle_{I_{\ell}}^{-1} \left \langle V g \right \rangle_{I_{\ell}}  \textbf{1}_{I_{\ell}}  \right \|_{L^2(V)}  \\
& \le A_1 \left \| W(I_k)^{\frac{1}{2}} \left \langle W \right \rangle_{I_k}^{-1}  \left \langle W f \right \rangle_{I_k} \right \|_{\mathbb{C}^d} \left \| V(I_{\ell})^{\frac{1}{2}}\left \langle V \right \rangle_{I_{\ell}}^{-1} \left 
\langle V g \right \rangle_{I_{\ell}} \right \|_{\mathbb{C}^d}\\
& = A_1 |I_k|^{\frac{1}{2}}
\left \| \left \langle W \right \rangle_{I_k}^{-\frac{1}{2}}  
\left \langle Wf \right \rangle_{I_k} \right \|_{\mathbb{C}^d}  
|I_{\ell}|^{\frac{1}{2}} \left \| \left \langle V \right \rangle_{I_{\ell}}^{-\frac{1}{2}} \left 
\langle V g \right \rangle_{I_{\ell}} \right \|_{\mathbb{C}^d} \\
& \le A_1C(d) \|f \|_{L^2(W)}  \| g \|_{L^2(V)} ,
\end{aligned}
\]
by Lemma \ref{lem:expectbd}. This immediately implies the desired bound for  $\left \langle T_W f_2 , g_2  \right \rangle_{L^2(V)}.$ The mixed terms are similarly straightforward. Specifically, observe that 
\[ \left |  \left \langle T_W f_2 , g_1  \right \rangle_{L^2(V)}  \right | \le    \| g_1 \|_{L^2(V)}  \sum_{k=1}^2 \left \| T_W E^W_{I_k}f\right \|_{L^2(V)}  \le A_1 C(d) \|f \|_{L^2(W)} \| g \|_{L^2(V)} ,\]
using the arguments that appeared in the previous bound. Similarly, 
\begin{eqnarray*}
\left | \left \langle T_W f_1 , g_2  \right \rangle_{L^2(V)} \right | =  \left |\left \langle  f_1 ,T^*_V g_2  \right \rangle_{L^2(W)} \right | & \le &  \| f_1 \|_{L^2(W)}  \sum_{\ell =1}^2 \left \| T_V^* E^V_{I_{\ell}} g \right \|_{L^2(W)}\\
& \le & A_2 C(d) \|f \|_{L^2(W)} \| g \|_{L^2(V)} ,
\end{eqnarray*}
using Lemma \ref{lem:expectbd} and the testing condition on $T^*_V.$ This completes the proof.
\end{proof}

We now turn to the proof of Theorem \ref{thm:WellLoc2}.

\begin{proof} This theorem is established in basically the same manner as Theorem  \ref{thm:WellLoc}. We simply need to check that the weaker conditions $(i)$ and $(ii)$ in Theorem \ref{thm:WellLoc2} allow us to deduce the same estimates. As before, we establish boundedness by duality as in \eqref{eqn:opineq}, fix $f,g$ compactly supported in $I_1 \cup I_2$ with $|I_1|=|I_2|=2^m,$ and decompose
\[ f = f_1 +f _2 \ \text{ and } g = g_1 + g_2 \]
as in \eqref{eqn:fdecomp} and \eqref{eqn:gdecomp}. As before, 
\[ \begin{aligned}
\left \langle T_W f_1, g_1 \right \rangle_{L^2(V)} &=  \left \langle \Pi^W f_1, g_1 \right \rangle_{L^2(V)} + \left \langle  f_1, \Pi^V g_1 \right \rangle_{L^2(W)}  \\
& \ \ \ \ +  \sum_{\substack{I: |I| \le 2^m \\ 1 \le i \le d}} \sum_{\substack{J: |J| \le 2^m \\ 2^{-r}|I| \le  |J| \le 2^{r} |I| \\ 1 \le j \le d}}\left \langle f, h^{W,i}_I \right \rangle_{L^2(W)} \left \langle g, h^{V,j}_J \right \rangle_{L^2(V)} \left \langle T_W h^{W,i}_I, h^{V,j}_J \right \rangle_{L^2(V)}. 
\end{aligned} \]
The first two terms can be controlled by testing hypothesis $(i)$ and Lemma \ref{lem:parbdd}. For the sum, we can use Lemma \ref{lem:weightedbdd} and testing hypothesis $(ii)$ to conclude
\[ \left | \left \langle T_W h^{W,i}_I, h^{V,j}_J \right \rangle_{L^2(V)} \right |  \le C(d) A_3. \]
Since $T_W$ is still well-localized with radius $r$, we can use the strategy from the proof of Theorem \ref{thm:WellLoc} to immediately conclude:
\[ \left| \left \langle T_W f_1, g_1 \right \rangle_{L^2(V)} \right |  \le 2^{2r} C(d) \left(A_1B(W) + A_2 B(V) +A_3 \right) \| f \|_{L^2(W)} \| g \|_{L^2(V)}.  \]
The other terms are also straightforward. First observe that since $|I_k | = |I_{\ell}|$, assumption $(ii)$ paired with Lemma \ref{lem:expectbd}  implies that for each $k, \ell$:
\begin{align} \nonumber
\left | \left \langle T_W E^W_{I_k} f , E^V_{I_{\ell}} g  \right \rangle_{L^2(V)}  \right | & \le  A_3
 \left \| W(I_k)^{\frac{1}{2}} \left \langle W\right \rangle^{-1}_{I_k} \left \langle Wf \right \rangle_{I_k} \right \|_{\mathbb{C}^d} 
 \left \| V(I_{\ell})^{\frac{1}{2}} \left \langle V\right \rangle^{-1}_{I_{\ell}} \left \langle Vg \right \rangle_{I_{\ell}} \right \|_{\mathbb{C}^d} \nonumber \\
 & = A_3 |I_k|^{\frac{1}{2}} \left \| \left \langle W \right \rangle_{I_k}^{-\frac{1}{2}} \left \langle Wf \right \rangle_{I_k} \right
\|_{\mathbb{C}^d}|I_{\ell}|^{\frac{1}{2}} \left \| \left \langle V \right \rangle_{I_{\ell}}^{-\frac{1}{2}} \left \langle Vg \right \rangle_{I_\ell} \right
\|_{\mathbb{C}^d} \nonumber \\
&\le A_3 C(d) \| f \|_{L^2(W)} \| g \|_{L^2(V)}. \label{eqn:est1}
\end{align}
This immediately gives the desired bound for $\left \langle T_W f_2, g_2 \right \rangle _{L^2(V)}$. The mixed terms require a bit more work. We consider $\left \langle T_W f_2, g_1 \right \rangle _{L^2(V)}$. The other term can be handled analogously. Observe that
\begin{eqnarray} \nonumber
\left| \left \langle T_W f_2, g_1 \right \rangle _{L^2(V)} \right|
&\le& \displaystyle  \sum_{k=1}^2  \sum_{\substack{J: |J| \le 2^m \\ 1 \le j \le d}} \left | \left \langle g, h^{V,j}_J \right \rangle_{L^2(V)} \left \langle T_WE^W_{I_k} f, h^{V,j}_J  \right \rangle_{L^2(V)} \right | \\
 \label {eqn:sum1} &=&\displaystyle  \sum_{k=1}^2  \sum_{\substack{J: J \subseteq I_k \\ 1 \le j \le d}} \left | \left \langle g, h^{V,j}_J \right \rangle_{L^2(V)} \left \langle T_WE^W_{I_k} f, h^{V,j}_J  \right \rangle_{L^2(V)} \right | \\
 \label{eqn:sum2}&& \ \ \ \  +\displaystyle  \sum_{k=1}^2  \sum_{\substack{J: |J| \le 2^m, J \not \subset I_k \\ 1 \le j \le d}} \left | \left \langle g, h^{V,j}_J \right \rangle_{L^2(V)} \left \langle T_WE^W_{I_k} f, h^{V,j}_J  \right \rangle_{L^2(V)} \right |. 
\end{eqnarray}
We have to handle \eqref{eqn:sum1} and \eqref{eqn:sum2} separately. To handle \eqref{eqn:sum1}, simply use Cauchy-Schwarz, Lemma \ref{lem:expectbd}, and assumption $(i)$ to conclude
\[
\begin{aligned}
\displaystyle  \sum_{k=1}^2  \sum_{\substack{J: J \subseteq I_k \\ 1 \le j \le d}} \left | \left \langle g, h^{V,j}_J \right \rangle_{L^2(V)} \left \langle T_WE^W_{I_k} f, h^{V,j}_J  \right \rangle_{L^2(V)} \right | 
&\le  \sum_{k=1}^2  \left \| \textbf{1}_{I_k}  T_WE^W_{I_k} f \right \|_{L^2(V)} \left \| \textbf{1}_{I_k} g \right \|_{L^2(V)} \\
& \le A_1  \| g \|_{L^2(V)}  \sum_{k=1}^2  |I_k|^{\frac{1}{2}} \left \| \left \langle W \right \rangle_{I_k}^{-\frac{1}{2}} \left \langle Wf \right \rangle_{I_k} \right
\|_{\mathbb{C}^d} \\
& \le A_1 C(d) \| f \|_{L^2(W)} \| g \|_{L^2(V)}.
\end{aligned}
\]
Now, consider \eqref{eqn:sum2}. Since $T_W$ is well-localized with radius $r$, one can easily that show that for each $I_k$, there are at most a fixed constant times $2^{2r}$ intervals $J$ that satisfy
\[  \left \langle T_WE^W_{I_k} f, h^{V,j}_J  \right \rangle_{L^2(V)} \ne 0,\]
$|J| \le 2^m,$ and $J \not \subset I_k.$  Indeed, for the inner product  to be nonzero, $J$ must satisfy $J \subset I_k^{(r+1)}$ and $|J| >2^{-r}|I_k|.$ Now,
using assumption $(ii),$ Lemma \ref{lem:expectbd}, and Lemma \ref{lem:haarbound}, we can establish the following sequence:
\[  
\begin{aligned}\displaystyle   \eqref{eqn:sum2}
&=  \sum_{k=1}^2  \sum_{\substack{J: 2^{-r} |I_k| < |J| \le |I_k| \\ J \subset I_k^{(r+1)} J \not  \subset I_k \\ 1 \le j \le d}} \left | \left \langle g, h^{V,j}_J \right \rangle_{L^2(V)} \left \langle T_WE^W_{I_k} f, h^{V,j}_J  \right \rangle_{L^2(V)} \right | \\
& \le \| g \|_{L^2(V)} \sum_{k=1}^2  \sum_{\substack{J: 2^{-r} |I_k| \le |J| \le |I_k| \\ J \subset I_k^{(r+1)},\, J \not  \subset I_k \\ 1 \le j \le d}} \left |  \left \langle T_WE^W_{I_k} f, h^{V,j}_J  \right \rangle_{L^2(V)} \right |\\
& \le A_3 \| g \|_{L^2(V)} \sum_{k=1}^2  \sum_{\substack{J: 2^{-r} |I_k| < |J| \le |I_k| \\ J \subset I_k^{(r+1)},\, J \not  \subset I_k \\ 1 \le j \le d}} |I_k|^{\frac{1}{2}} \left \| \left \langle W \right \rangle_{I_k}^{-\frac{1}{2}} \left \langle Wf \right \rangle_{I_k} \right
\|_{\mathbb{C}^d} \left \| V(I_-)^{\frac{1}{2}} h_J^{V,j}(J_-) \right \|_{\mathbb{C}^d} \\
&\ \ \ \ +   A_3 \| g \|_{L^2(V)} \sum_{k=1}^2  \sum_{\substack{J: 2^{-r} |I_k| \le |J| \le |I_k| \\ J \subset I_k^{(r+1)},\, J \not  \subset I_k \\ 1 \le j \le d}} |I_k|^{\frac{1}{2}} \left \| \left \langle W \right \rangle_{I_k}^{-\frac{1}{2}} \left \langle Wf \right \rangle_{I_k} \right
\|_{\mathbb{C}^d}
\left \| V(J_+)^{\frac{1}{2}} h_J^{V,j}(J_+) \right \|_{\mathbb{C}^d} \\
& \le 2^{2r} C(d) A_3  \|g \|_{L^2(V)} \| f \|_{L^2(W)},
\end{aligned} 
 \]
which completes the proof.
\end{proof}

\begin{remark}  \label{remark:definition} As mentioned earlier, our definition of well-localized is slightly different than the one appearing in \cite{ntv08}, where Nazarov-Treil-Volberg only impose conditions on $T_W$ when $|J| \le  |I|,$ rather than $|J| \le 2 |I|.$ The difference is likely attributable to a typographical error and their ideas are essentially correct.

However, to see why imposing conditions on  only $|J| \le |I|$ is not quite sufficient, let us consider the role of the well-localized property in the proofs of Theorems \ref{thm:WellLoc} and \ref{thm:WellLoc2}. 
It is used to show that for each fixed $I$, there is at most a finite number of $J$ with $2^{-r}|I| \le  |J| \le 2^{r} |I|$ such that 
\[ \left |  \left \langle T_W h^{W,i}_I, h^{V,j}_J \right \rangle_{L^2(V)}  \right | \ne 0.\]
This allows one to control related sums given in \eqref{eqn:diagonal2}. However, the definition of well-localized given by Nazarov-Treil-Volberg is not quite enough for this, as it does not handle the case where $|I|=|J|.$ In this case, one would need control over terms such as
\[  \left |  \left \langle T_W h^{W,i}_I(I_+) \textbf{1}_{I_+}, h^{V,j}_J \right \rangle_{L^2(V)} \right| \text{ or }  \left |  \left \langle  h^{W,i}_I, T_V^* h^{V,j}_J(J_+)\textbf{1}_{J_+} \right \rangle_{L^2(W)}   \right|, \]
which are not  addressed in their definition of well-localized since $|I_+| <|J|$ and $|J_+| <  |I|.$  This case is no longer a problem if we impose conditions on all $I,J$ with $|J| \le 2|I|$ as in Definition \ref{def:wellloc}.
For an example of what can go wrong, fix $K_0 \in \mathcal{D}$. Fix a sequence $\{c_K\}$ in $\ell^2(\mathcal{D})$ with no nonzero terms, and define the operator $T: L^2(\mathbb{R}) \rightarrow L^2(\mathbb{R})$ by
\[ T h_{K_0} \equiv \sum_{K: |K|=|K_0|} c_K h_K  \ \text{ and } \ \  T h_L  \equiv 0 \text{ for } L \ne K_0.  \]
It is not difficult to show $T$ is well-localized (with radius $0$) from $L^2(\mathbb{R})$ to $L^2(\mathbb{R})$ according to the definition in \cite{ntv08}. Indeed, if $|J| \le |I|$, then 
\[ \left \langle T \textbf{1}_I, h_J \right\rangle_{L^2} = 0= \left \langle T^* \textbf{1}_I, h_J \right \rangle_{L^2}.\]
To see these equalities, first write 
\[ \textbf{1}_I =\sum_{K: I \subsetneq K} \left \langle \textbf{1}_I, h_K \right \rangle_{L^2} h_K. \]
Thus, if $I$ is not strictly contained in $K_0$, then $T\textbf{1}_I =0.$ So, we can assume $I \subsetneq K_0.$ Then $|J| \le |I| < |K_0|$ so 
\[\left  \langle T \textbf{1}_I, h_J \right \rangle_{L^2} = \sum_{K:|K| = |K_0|} \left \langle \textbf{1}_I, h_{K_0} \right \rangle_{L^2} c_K \left \langle h_K, h_J \right \rangle_{L^2} =0.\]
Now consider $T^*.$ If $|J| \le |I|$ and $J \ne K_0$, then
\[ \left \langle T^* \textbf{1}_I, h_J \right \rangle_{L^2} = \left \langle  \textbf{1}_I, T h_J \right \rangle_{L^2} = \left \langle  \textbf{1}_I, 0 \right \rangle_{L^2}=0 \]
immediately. If $J =K_0$, then
\[ \left \langle T^* \textbf{1}_I, h_J \right \rangle_{L^2} = \sum_{K: K=|K_0|} \overline{c_K}  \left \langle  \textbf{1}_I, h_K \right \rangle_{L^2} =0,\]
since $|K_0|= |J| \le |I|$ implies $K \subseteq I$ or $K \cap I =0.$
However, for this operator $T$, 
\[ \left \langle T h_{K_0}, h_J \right \rangle_{L^2} =c_J \ne 0, \]
for all $J$ with  $|J| = |K_0|$. Since there are is infinite number of such $J$, this means  we could not use the well-localized property to control the sums from \eqref{eqn:diagonal2} for this operator.
\end{remark}

\begin{remark} In this paper, we only considered band operators defined on $L^2(\mathbb{R}, \mathbb{C}^d).$  However, we anticipate that these T1 theorems will generalize without substantial difficulty to band operators on $L^2(\mathbb{R}^n, \mathbb{C}^d)$. One must define a slightly more complicated Haar system, but in general, the tools and proof strategy seem to work without issue. 
\end{remark}

\end{document}